\documentclass[12pt]{article}
\usepackage[dvipsnames]{xcolor}
\usepackage{amsmath,amsthm,amssymb,enumerate}
\usepackage{graphicx}
\usepackage[backref=page]{hyperref}
\usepackage[all]{xy}
\usepackage{ytableau}
\usepackage{palatino}
\usepackage[mathbf]{euler}
\usepackage{anyfontsize} 
\usepackage{subcaption}
\usepackage{soul}
\usepackage{cleveref}
\usepackage[vcentermath]{youngtab}
\hypersetup{colorlinks=true,linkcolor=green!50!black,citecolor=red!50!black}
\usepackage{titlesec}
\titleformat*{\section}{\fontseries{b}\fontsize{16}{20}\selectfont}
\titleformat*{\subsection}{\fontseries{b}\fontsize{13}{17}\selectfont}
\usepackage{abstract}
\usepackage{siunitx}
\usepackage{tikz-cd}
\usepackage{tikz}

\usepackage{mathtools}
\usepackage{adjustbox}

\makeatletter
\def\Ddots{\mathinner{\mkern1mu\raise\p@
\vbox{\kern7\p@\hbox{.}}\mkern2mu
\raise4\p@\hbox{.}\mkern2mu\raise7\p@\hbox{.}\mkern1mu}}
\makeatother

\newcommand{\bv}{\mathbf v}
\newcommand{\bw}{\mathbf w}

\newcommand{\bn}{\mathbf n}
\newcommand{\Cx}{{\mathbb C^\times}}
\newcommand{\Irr}{\operatorname{Irr}}
\newcommand\Rpp{\operatorname{RPP}}

\newcommand\soc{\operatorname{soc}}
\newcommand{\dimvec}{\underrightarrow{\dim}\,}

\newcommand{\new}[1]{\textbf{#1}}
\newcommand{\fg}{\mathfrak{g}}
\newcommand{\fh}{\mathfrak{h}}

\newcommand{\bC}{\mathbb{C}}

\newcommand{\bN}{\mathbb{N}}

\newcommand{\sslash}{\mathbin{/\mkern-6mu/}}
\DeclareMathOperator{\Hom}{Hom}

\DeclareMathOperator{\Gr}{Gr}
\newcommand\restr[2]{{
  \left.\kern-\nulldelimiterspace 
  #1 
  \vphantom{\big|} 
  \right|_{#2} 
  }}

\newtheorem{theorem}{Theorem}[section]
\newtheorem*{theorem*}{Theorem}
\newtheorem{proposition}[theorem]{Proposition}
\newtheorem*{proposition*}{Proposition}
\newtheorem{conjecture}[theorem]{Conjecture}
\newtheorem*{conjecture*}{Conjecture}
\newtheorem{lemma}[theorem]{Lemma}
\newtheorem*{lemma*}{Lemma}
\newtheorem{corollary}[theorem]{Corollary}
\newtheorem*{corollary*}{Corollary}
\newtheorem{definition}[theorem]{Definition}
\newtheorem*{definition*}{Definition}
\newtheorem{example}[theorem]{Example}
\newtheorem*{example*}{Example}

\newtheorem*{exercise*}{Exercise}
\newtheorem{remark}[theorem]{Remark}
\newtheorem*{remark*}{Remark}

\newtheorem*{question*}{Question}

\newtheorem*{claim*}{Claim}

\title{Heaps, Crystals, and Preprojective Algebra Modules}
\author{Anne Dranowski, Bal\'{a}zs Elek, Joel Kamnitzer, \\ and Calder Morton-Ferguson}
\date{\today}

\begin{document}
  
\maketitle
\begin{abstract}

Fix a simply-laced semisimple Lie algebra.  We study the crystal $ B(n\lambda)$, were $\lambda$ is a dominant minuscule weight and $ n $ is a natural number.  On one hand, $B(n\lambda)$ can be realized combinatorially by height $ n $ reverse plane partitions on a heap associated to $ \lambda $. On the other hand, we use this heap to define a module over the preprojective algebra of the underlying Dynkin quiver.  Using the work of Saito and Savage-Tingley, we realize $B(n\lambda)$ via irreducible components of the quiver Grassmannian of $n$ copies of this module. In this paper, we describe an explicit bijection between these two models for $B(n\lambda)$ and prove that our bijection yields an isomorphism of crystals.  Our main geometric tool is Nakajima's tensor product quiver varieties. 

\end{abstract}
\setcounter{tocdepth}{2}
\tableofcontents

\section{Introduction}
\label{sec:intro}

Let $\fg$ be a simply-laced complex semisimple Lie algebra and let $V(\lambda)$ be the irreducible representation of $\fg$ of highest weight $\lambda\in P_+$ (the set of dominant weights). The crystal $B(\lambda)$ of $V(\lambda)$, introduced by Kashiwara in \cite{Ka91}, is a coloured directed graph encoding the combinatorics of $V(\lambda)$. It has many different realizations. 
These include semistandard Young tableaux in type $A$ and their analogues in type $D$ \cite{KN94}, Littelmann paths \cite{L95}, Mirkovi\'c--Vilonen polytopes \cite{kamnitzer2007crystal}, and Kashiwara's monomials \cite{kashiwara2003realizations}, on the combinatorial side.  
On the geometric side, we can consider Springer fibres \cite{malkin2002tensor}, Mirkovi\'c--Vilonen cycles \cite{braverman2001crystals}, irreducible components of cores of Nakajima quiver varieties \cite{Sai02}, and simple perverse sheaves on quiver varieties \cite{lusztig1998quiver}.

\subsection{Young tableaux and Springer fibres} 
\label{subsec:ytab}
When $ \fg = \mathfrak sl_m $ dominant weights can be viewed as partitions, and $ B(\lambda) $ can be identified with the set $ SSYT(\lambda) $ of semistandard Young tableaux of shape $ \lambda $ with labels $\{1,\dots,m\}$, or, in a compatible way, with the set of irreducible components of $m$-step Springer fibres. To motivate our main result, let us briefly recall these models, and how they are related. 

Let $ A : \bC^N \rightarrow \bC^N $ be a nilpotent linear operator whose Jordan type is given by $ \lambda $.  
Consider the 
space of all $m$-step flags which are preserved by $A$,
$$ 
  F(A) := \left\{ 0= V_0 \subseteq V_1 \subseteq \cdots \subseteq V_m = \bC^N \,\big|\, A V_i \subseteq V_{i-1} \right\} . 
$$
There is a crystal structure on the set $ \Irr F(A)$ of irreducible components of $F(A)$ making it isomorphic to $ B(\lambda) $, with raising and lowering operators acting via Ginzburg's correspondences \cite[Chapter 4]{chriss1997representation}, as described by Malkin \cite[Section 3.7]{malkin2002tensor}.
The connected components of $F(A) $ correspond to
weight spaces of $ B(\lambda) $.

Given a flag $V = (V_0,V_1,\dots,V_m)$ in $F(A) $, we can encode the Jordan types of the restrictions $A|_{V_i}$ for $1\le i\le m$ in a tableau.  
More precisely, let $ \Psi_V $ be the tableau such that the shape formed by the boxes with labels in $\{1, \dots, i\}$ in $\Psi_V$ is the transpose of the Jordan type of $ A|_{V_i} $. 
The map $ V \mapsto \Psi_V$ is constructible on $ F(A)$, so we can define $\Psi_Z$ to be its value on a generic point of the irreducible component $Z$.  
This gives a bijection by work of Spaltenstein \cite{spaltenstein1976fixed}, and a crystal isomorphism by work of Savage \cite{Sav06}.
\begin{theorem} \label{th:classic}
  The map $$ \Irr F(A) \rightarrow SSYT(\lambda) \quad Z \mapsto \Psi_Z $$ is an isomorphism of crystals.
\end{theorem}

\subsection{Reverse plane partitions}
\label{subsec:rpp}

Our goal in this paper is to generalize \Cref{th:classic} in a type-independent way, using the combinatorics of reverse plane partitions and the geometry of quiver Grassmannians.

Recall that a weight $\lambda\in P_+$ is \new{minuscule} if the Weyl group $W$ of $\fg$ acts transitively on the weights of $V(\lambda)$. 
Fix $\lambda$ minuscule, denote by $w_0$ the longest element of $W$, and let $ w \in W $ be the minimal element such that $ w\lambda = w_0 \lambda $.
Associated to $ w $ is a poset $ H(w) $, called the \textbf{heap} of $ w $, which encodes all reduced words for $ w$ (see \cref{sec:heaps}). 
 
Since $ \lambda $ is minuscule, it is very easy to make a combinatorial model for the crystal 
$ B(\lambda) $.  The orbit $ W \lambda $ is one such model. The set $ \left\{ v \in W : v \le_L w \right\} $ of terminal subwords of $w$, (demarcated by $\le_L$ the left weak Bruhat order,) is another.  
In this paper, we will use the set of order ideals $J(H(w))$ in $H(w)$ as our primary model for $ B(\lambda)$ (see \Cref{prop:WJcrystal}).

Our main object of study will be the following set of order-reversing maps on $H(w)$.
\[
  \Rpp(w,n) :=  \left\{\Phi: H(w)\to \left\{0,\ldots,n\right\} \,\big|\, \Phi(x) \ge \Phi(y) \text{ whenever } x \le y \right\}\,. 
\]
Elements of $\Rpp(w,n)$ are also called reverse plane partitions (of shape $H(w)$ and height $n$) or simply RPPs. As such, these first appear in work of Proctor \cite{proctor1984bruhat}.

Let $ \phi_k := \Phi^{-1}\left(\{n - k +1, \dots, n\}\right)$. 
We consider the natural inclusion $ \Rpp(w,n) \rightarrow J(H(w))^n$ given by $ \Phi \mapsto (\phi_1, \dots, \phi_n) $. 
We show that there is a crystal structure on $ \Rpp(w,n) $ that is compatible with this inclusion.
\begin{theorem}
	There is an isomorphism of crystals $ \Rpp(w,n) \cong B(n \lambda) $ making the following diagram commute. 
  \[
  \begin{tikzcd}
    RPP(w,n) \ar[r]\ar[d] & B(n\lambda) \ar[d]\\
    J(H(w))^n \ar[r] & B(\lambda)^{\otimes n} 
  \end{tikzcd}  
  \]
\end{theorem}
On a purely numerical level (i.e.\ ignoring the crystal structure) this can be deduced from \cite[Theorem 8]{proctor1984bruhat}.  On the other hand, we can easily relate these RPPs to Lakshmibai--Seshadri paths \cite{LS91} (see \Cref{re:LS}) and the crystal structure on these paths was studied by Littelmann \cite{L94}.

\subsection{Quiver Grassmannians} \label{subsec:quivergr}
Keep $\lambda\in P_+$ and $ w\in W$ as above. 
Let $ \Pi $ be the preprojective algebra of $\fg$.
On the geometric side, we will work with a model for the crystal $ B(n\lambda) $ defined using quiver Grassmannians of injective $ \Pi$-modules.

From the heap $H(w)$, we construct a module $ \bC H(w)$ for $\Pi$. This module can be visualized using the Kleshchev--Ram glass bead game for the heap \cite{Ram15}. 
Submodules of $ \bC H(w) $ are in natural bijection with order ideals of $ H(w) $ or right subwords $\left\{v\le_L w\right\}$. 

We define the quiver Grassmannian of $\bC H(w)$ denoted $\Gr(\bC H(w)^{\oplus n}) $ to be the set of all $\Pi$-submodules of $ \bC H(w)^{\oplus n} $.
This is a disconnected projective variety whose connected components are labelled by possible dimension vectors of submodules.
We prove that $ \bC H(w) $ is an injective $ \Pi$-module.
By \cite[Prop.~4.12]{savage2011quiver} it follows that $ \Gr(\bC H(w)^{\oplus n}) $ is the core of a Nakajima quiver variety, and so by \cite[Th.~4.6.4]{Sai02}, the set of irreducible components of $ \Gr(\bC H(w)^{\oplus n}) $ is a geometric realization of $ B(n\lambda) $.   

\subsection{From submodules to reverse plane partitions} \label{subsec:mod-to-rpp}

As both $\Rpp(w,n) $ and $ \Irr \Gr(\bC H(w)^{\oplus n})$ are realizations of $ B(n\lambda) $, there is a unique crystal isomorphism between them.  
Our goal in this section is to describe the underlying bijection.

The heap $H(w)$
is equipped with a map $\pi:H(w)\to I$ \cite[Section~2.2]{St96}, such that the sets $H(w)_i := \pi^{-1}(i)$ are totally ordered for any $i\in I$.   So for each $ i \in I$, we can define a nilpotent operator $ A_i $ on each $ \bC H(w)_i $ as a Jordan block in the basis $H(w)_i$. 
This definition extends to a nilpotent operator (also denoted) $A_i$ on $\bC H(w)^{\oplus n}_i $. By the construction of $ A_i$, if $M \subset \bC H(w)^{\oplus n} $ is a submodule, then $M_i$ is invariant under $ A_i$. So given $M\in \Gr(\bC H(w)^{\oplus n})$, we can consider the Jordan type of $ A_i\big|_{M_i} $.  We encode these Jordan types into a reverse plane partition $ \Phi_M $ as follows. The values of $ \Phi_M $ on $ H(w)_i $ form the transpose of the partition giving the sizes of the Jordan blocks of $ A_i\big|_{M_i} $ for each $i$. In section \cref{se:socle}, we show that $ \Phi_M $ can also be regarded as encoding the socle filtration of $ M $.

As in \cref{subsec:ytab}, since $\Phi$ is a constructible function on $Gr(\mathbb{C}H(w))$, for any irreducible component $Z$ of $Gr(\mathbb{C}H(w)^{\oplus n})$, we write $ \Phi_Z$ for the value of $ \Phi $ on a generic point of $Z$. Here is the main result of this paper.
\begin{theorem} \label{thm:main}
  The map $ Z \mapsto \Phi_Z $ defines the crystal isomorphism $ \Irr \Gr(\bC H(w)^{\oplus n}) \cong \Rpp(w,n)$.
\end{theorem}
This theorem has the same flavour as \Cref{th:classic}.  In fact, when $ \fg = \mathfrak{sl}_m $ the quiver Grassmannian $ \Gr(\bC H(w)^{\oplus n}) $ is isomorphic to a Springer fibre (associated to a rectangular Jordan type) and \Cref{thm:main} reduces to \Cref{th:classic} (see \cref{se:Maffei}).

We now describe our method of proof of \Cref{thm:main}. Any $ M \subset \bC H(w)^{\oplus n} $ carries a filtration by $ M^{\le k} := M \cap \bC H(w)^{\oplus k}$ for $k = 0,1,\dots,n$. Each subquotient $ M^k := M^{\le k}/M^{\le k-1}$ of this filtration is a submodule of $ \bC H(w) $ and thus can be identified with an order ideal $ \phi_k $ in $ H(w) $.

We use Nakajima's tensor product varieties to prove that an irreducible component $Z$ of $Gr(\bC H(w)^{\oplus n})$ corresponds to $ \phi_1 \otimes \cdots \otimes \phi_n $ under the embedding 
$$
  B(n\lambda) \hookrightarrow B(\lambda)^{\otimes n}
$$ 
if a general point $M \in Z $ yields the subquotients $\bC\phi_1, \dots, \bC\phi_n \subset \bC H(w)$.

On the other hand, we prove that the information of these subquotients is enough to find the Jordan types of the nilpotent operators used to construct \(\Phi_M\) and therefore deduce the irreducible component giving rise to them. This requires a result of Stembridge on the structure of minuscule heaps, as well as some case by case analysis.
\begin{remark}
  In the body of the paper, we work more generally with any dominant minuscule $ w \in W $ with dominant witness $ \lambda$.  In this more general setup, $ \Irr \Gr(\bC H(w)^{\oplus n}) $ and $\Rpp(w,n) $ give models for the Demazure crystal $B_w(n\lambda) $.
\end{remark}

\subsection{Future directions} \label{subsec:future}

This paper is part of a larger project begun in \cite{mvbasis} whose goal is to compare the dual canonical, dual semicanonical, and MV bases for representations of $ \fg $.  These bases are indexed by simple KLR modules, generic preprojective algebra modules, and Mirkovi\'c--Vilonen cycles, respectively.  
From this perspective, it would be natural to try to describe the simple KLR modules and MV cycles labelled by these RPPs.  
We hope that, as in this paper, this can be done in a type-independent fashion.  As a first step, we note that the simple KLR modules associated to order ideals in $ H(w) $ are the strongly homogeneous modules of Kleshchev--Ram \cite{KR10}. 
Similarly, the MV cycles associated to order ideals in $ H(w) $ are Schubert varieties in the partial flag variety $ G/P_\lambda $, where $P_\lambda$ is the parabolic subgroup associated to the minuscule dominant weight $\lambda $.

In another direction, we know that $\bigoplus_n V(n\lambda)$ is the homogeneous coordinate ring of the partial flag variety $ G/P_\lambda$. 
This coordinate ring carries a cluster algebra structure \cite{geiss2008partial}.  It would be interesting to understand the interaction of the combinatorics of the cluster algebra with the combinatorics of $ \bigsqcup_n \Rpp(w,n)$. For example, which collections of RPPs correspond to clusters and what does mutation look like in this combinatorics? 

Our work was inspired by the paper of Garver--Patrias--Thomas \cite{GPT18} who studied certain quiver representations and reverse plane partitions.  In their work, RPPs encode Jordan types of nilpotent endomorphisms of quiver representations.  This looks quite similar to their use in our paper, but we were not able to make any precise connection between their work and ours.  It would be interesting to investigate this further.

\subsection{Acknowledgments}
We would like to thank Pierre Baumann, Elie Casbi, Steven Karp, Peter Tingley, and Hugh Thomas for helpful conversations on these topics. We would also like to thank Nathaniel Libman for significant contributions to the project through his suggestions and several useful discussions.

\section{Heaps and crystals}
\label{sec:heaps}

Write $s_i$ for the simple reflections generating the Weyl group $W$ and $\alpha_i,\alpha_i^\vee$ for the associated simple (co)roots. 
\subsection{Heaps and dominant minuscule elements}
\label{sec:domminheaps}

In this section we lay the combinatorial foundations for the rest of the paper. The starting point for both our geometric and combinatorial constructions is a minuscule heap (\Cref{def:minuheap}). Our main references for this section are Stembridge's papers \cite{St96} and \cite{St01}. 

\begin{definition}\label{def:heap}
  Let $\mathbf{w}=(s_{i_1},\ldots , s_{i_\ell})$ be a reduced word for $w\in W$. The \textbf{heap} $H(\mathbf{w})$ of $\mathbf{w}$ is the partially ordered set $(\{1,2,\dots,\ell\},\prec)$ 
  defined as the transitive closure of the relation
  \[
    a \prec b \quad\text{ if }\quad a>b\text{ and }C_{{i_a},{i_b}} < 0,
  \]
  where $C_{{i_a},{i_b}}$ is the $(i_a,i_b)$ entry of the Cartan matrix of $\fg$.
\end{definition}
$H(\bw)$ is endowed with a map $\pi:H(\bw)\to I$ defined by $\pi(a)=i_a$ for each $a\in\{1,2,\dots,\ell\}$, and we denote $H(\bw)_i := \pi^{-1}(i)$. 

\begin{remark}\label{rmk:runners}
Since $  \mathbf w $ is reduced, any two consecutive occurrences of $ s_i $ in this word must be separated by a generator $s_j$ that does not commute with $ s_i$, so by transitivity each of the sets $H(\bw)_i$ are totally ordered.
\end{remark}
In thinking about $ H(\bw) $, we will use the  abacus analogy of Kleshchev and Ram \cite{KR10}. Imagine a wooden base in the shape of the Dynkin diagram, and attached to each node in the diagram is a vertical runner extending upward.   
We read the simple reflections in $\mathbf w$ from right to left, adding a bead to the $i_k$th rung of the abacus when $s_{i_k}$ is read. Subsequent beads fall to rest against beads already placed on neighbouring runners. See \Cref{fig:3412heap} for an example. Ram calls this visualization the ``glass bead game''.
\begin{figure}[ht]
  \centering
  \begin{tikzpicture}
    \foreach \x in {0,1,2,3} 
      \draw [line width=1mm] (\x, 0) -- (\x,5);
    \foreach \x in {0,1,2} 
      \shade[shading=ball, ball color=white] (\x+1,\x+1) circle (0.7);
    \foreach \x in {0,1,2} 
      \shade[shading=ball, ball color=white] (\x,\x+2) circle (0.7);
    \draw [line width=0.5mm] (0,0) -- (3,0);
    \foreach \x/\xtext in {0/1,1/2,2/3,3/4} 
      \draw (\x cm,0 pt) node[fill=white,inner sep=2pt,draw] {${\xtext}$};
  \end{tikzpicture}
  \caption{The heap $H((s_3, s_4, s_2, s_3, s_1, s_2))$ in type $A_4$}
  \label{fig:3412heap}
\end{figure}

By \cite[Proposition~2.2]{St96}, two heaps $H(\mathbf{w}_1)$ and $H(\mathbf{w}_2)$ give rise to isomorphic posets if $\mathbf{w}_2$ can be obtained from $\mathbf{w}_1$ by transposing adjacent commuting pairs of Coxeter generators. 
When any two reduced words for a given $w\in W$ are related in this way,
$w$ is called \textbf{fully commutative}. 
\begin{example}
  The element $s_2s_1s_3s_2$ in $S_4$ is fully commutative. In general, an element of $S_n$ is fully commutative if and only if it is $321$-avoiding (in one-line notation) if and only if no reduced expression for it contain $s_i s_{i+1} s_i$ as a subword.
\end{example}

If $w$ is fully commutative, then any two reduced words give rise to canonically 
isomorphic heaps, so we identify them as a single heap which is denoted $H(w)$. 

We will be interested in a special class of fully commutative elements. 
Recall that in the beginning of the introduction we assumed $\fg$ is simply-laced. This lets us simplify some of the statements in Stembridge's papers \cite{St96} and \cite{St01}. We state the simplified versions but cite the original ones.

\begin{definition}\cite{St01} \label{def:minuscule}
  Let $w \in W$ and $\lambda \in P $. If there exists a reduced word $(s_{i_1},\ldots, s_{i_\ell})$ for $w$ such that
  \[
    \langle \alpha_{i_k}^\vee  , s_{i_{k+1}}\cdots s_{i_{\ell -1}} s_{i_\ell}\lambda \rangle =1
  \]
  for all $1\leq k \leq \ell$, or, equivalently,
  \[
    s_{i_{k}}\cdots s_{i_\ell}\lambda = \lambda -\alpha_{i_\ell} -\alpha_{i_{\ell-1}}-\cdots  -\alpha_{i_k}
  \]
  for all $1\leq k \leq \ell$, then we say that $w$ is \textbf{$\lambda$-minuscule}. We say that $w$ is \textbf{dominant minuscule} if it is $\lambda$-minuscule for some dominant $\lambda$. 
  In either case, $\lambda$ is called a \new{witness} for $w$. 
\end{definition}

\begin{proposition}(\cite[Proposition 2.1]{St01}) 
  If $w$ is $\lambda$-minuscule then $w$ is fully commutative and the condition in \Cref{def:minuscule} holds for any choice of reduced word.
\end{proposition}

\begin{definition}\label{def:minuheap}
  If $w$ is minuscule (dominant minuscule), then we say that $H(w)$ is a minuscule (dominant minuscule) heap.
\end{definition}

Stembridge gives a characterization of minuscule and dominant minuscule Weyl group elements which we now recall. 
\begin{proposition}(\cite[Proposition 2.3]{St01})
  \label{thm:lminuscule}
  If $\mathbf{w}=(s_{i_1},\ldots ,s_{i_\ell})$ is a reduced word for $w$ then $w$ is minuscule if and only if between every pair of consecutive occurrences of a generator there are exactly two generators that do not commute with it. 
\end{proposition}

\begin{proposition}(\cite[Proposition 2.5]{St01})
  \label{dommin-condition}
  If $\mathbf{w}=(s_{i_1},\ldots ,s_{i_\ell})$ is a reduced word for $w$ then $w$ is dominant minuscule if and only if the conditions of \Cref{thm:lminuscule} are satisfied and the last occurrence of each generator in $\mathbf{w}$ is followed by at most one generator that does not commute with it.
\end{proposition}
The following result is also due to Stembridge. 
\begin{proposition}(\cite[Lemma 3.1]{St96})
  \label{thm:JH(w)weakorder}
  Let $w$ be fully commutative. Then $v \mapsto H(v)$ defines an isomorphism of posets $\left\{ v\in W : v\leq_L w \right\} \cong J(H(w))$.
\end{proposition}

Another result of Stembridge is the following.
\begin{proposition}(\cite[Corollary 3.4]{St01}) \label{prop:rank}
	Let $ w $ be minuscule.  Then $ H(w) $ is a ranked poset.
\end{proposition}
We will refer to the values of the rank function on $ H(w) $ as \textbf{levels}.  In terms of the abacus model for the heap, this agrees with the physical notion of level. For example in Figure \ref{fig:3412heap}, the bottom bead on runner 2 has level 1, the bead on runner 1 has level 2, the bead on runner 4 has level 3, and so on.

\subsection{Crystals} \label{subsec:crystals}

In this section, we will recall the notion of a crystal of a representation and fix some more notation, following \cite{HK02}. 
Let $\fg$ be a semisimple Lie algebra containing a Cartan subalgebra $\fh$, let $P$ be the weight lattice, and let $\langle \phantom{0},\phantom{0} \rangle$ be the canonical pairing between $\fh^\ast$ and $\fh$.

\begin{definition}(\cite[Section~4.5]{HK02})
  \label{def:crystals}
  An \new{upper semi-normal} $\fg$-\new{crystal} is a set $B$ together with 
  \begin{itemize}
  \item a map $wt:B\to P$,
  \item maps $\varepsilon_i:B\to\bN$ and $\varphi_i:B\to\bN$ for $i\in I$, and
  \item raising operators $e_i: B\to B\sqcup \{ 0\}$ and lowering operators $f_i:B\to B\sqcup \{ 0\}$ for each $i\in I$,
  \end{itemize}
  satisfying the following axioms:
  \begin{enumerate}
  \item for $b\in B$, $\varphi_i(b)=\varepsilon_i(b)+\langle \alpha_i^\vee , wt(b) \rangle$, 
  \item for $b\in B$ with $e_i(b)\in B$, $  wt(e_i(b))=wt(b)+\alpha_i$,
  \item for $b\in B$ with $f_i(b)\in B$, $ wt(f_i(b))=wt(b)-\alpha_i$,
  \item for $b_1,b_2\in B$, $f_i(b_2)=b_1$ if and only if $b_2 = e_i(b_1)$, and
  \item for $i\in I$, $\varepsilon_i(b) = \max \left\{ n\geq 0 : e_i^n(b)\neq 0 \right\}$. 
  \end{enumerate}
\end{definition}

Note that we do not complement the last axiom with an analogous one for the lowering operators since we will be studying Demazure crystals.

Suppose that $ B$ is an upper semi-normal crystal.  A subset $ B' \subset B $ will be called an \new{upper semi-normal subcrystal} if it is stable under raising operators.  More precisely, for all $ b \in B' $ and for all $ i \in I $, if $ e_i(b) \ne 0 $ in $ B$, then $ e_i(b) \in B'$.  (We do not require the same condition for lowering operators.)  In this case, $ B' $ inherits the structure of an upper semi-normal crystal, with $ wt, e_i, \varepsilon_i, \varphi_i $ defined by restriction from $ B $, and where we define $ f_i(b) = 0 $ whenever $ f_i(b) \notin B' $.  

In this paper, \new{crystal} will always mean ``upper semi-normal crystal'' and \new{subcrystal} will always mean ``upper semi-normal subcrystal''.

Every $ \fg$-representation has a crystal.
This crystal can either be defined by using a crystal basis for a representation of the quantum group \cite[Chapter 3]{HK02}, or by using a perfect basis for the representation \cite{BK2}.  

\begin{definition}\label{def:demazureops}
  Let $B(\lambda)$ be the crystal of the irrep $V(\lambda)$ with highest weight element $b_\lambda$ and let $w\in W$. Let $\mathbf{w}=(s_{i_1},\ldots ,s_{i_\ell})$ be a reduced word for $w$. The subcrystal
  \[
    B_w(\lambda):=\bigcup_{m_k\in \mathbb{N}} f_{i_1}^{m_1}f_{i_2}^{m_2}\cdots f_{i_\ell}^{m_\ell}b_\lambda \subset B(\lambda)
  \]
  is the \textbf{Demazure crystal}.
\end{definition}

Note that $ B_w(\lambda) $ is the crystal of the Demazure module $V_w(\lambda) $, the $\mathfrak b $-subrepresentation of $ V(\lambda) $ generated by $w\cdot v(\lambda) $ \cite{Ka93} where $v(\lambda)$ is the highest weight element. 

\subsection{Tensor products of crystals}\label{subsec:tensorcrystals}
The most important feature of crystals that we will use is the tensor product rule. We recall it here to fix our conventions.

\begin{definition}(\cite[Theorem~4.4.1]{HK02})
  Let $B_1,B_2$ be crystals. Then the set 
  $$
  B_1\otimes B_2=\left\{b_1\otimes b_2 \; |\; b_1\in B_1, b_2\in B_2\right\}\,,
  $$
  with lowering operators
  \begin{align*}
    f_i(b_1\otimes b_2) &=
                          \begin{cases}
                            b_1\otimes f_i(b_2) &\text{ if }\varepsilon_i(b_1) < \varphi_i(b_2)\\
                            f_i(b_1)\otimes b_2 &\text{ if }\varepsilon_i(b_1)\geq \varphi_i(b_2)
                          \end{cases}                                                  
  \end{align*}
  (where we declare $b_1\otimes 0 = 0 \otimes b_2 =0$) and functions
  \begin{align*}
    wt(b_1\otimes b_2)&=wt(b_1)+wt(b_2), \\
    \varepsilon_i(b_1\otimes b_2) &= \max\left( \varepsilon_i(b_1),\varepsilon_i(b_2)-\langle \alpha_i^\vee, wt(b_1) \rangle \right),\\
    \varphi_i(b_1\otimes b_2) &= \max\left( \varphi_i(b_2),\varphi_i(b_1)+\langle \alpha_i^\vee, wt(b_2) \rangle \right),
  \end{align*}
  is a crystal.
\end{definition}
The tensor product rule has an extension to tensor products with more than two tensor factors known as the \textbf{signature rule}. We recall it here.

\begin{corollary}\label{thm:signrule}
  (\cite{tingley2008three})
The set $\bigotimes_{j=1}^n B_j$ consists of tuples  
$b_1\otimes b_2 \otimes \ldots \otimes b_n$ with $b_j\in B_j$ for $j = 1,\dots,n$. 
To apply a lowering (resp.\ raising) operator $f_i$ (resp.\ $e_i$) to such an element, we first compute its sign pattern: 
working our way from left to right, for each factor $b_j$, we record $\varphi_i(b_j)$ many $+$ signs, followed by $\varepsilon_i(b_j)$ many $-$ signs. 
We then cancel all $-+$ pairs getting a sequence of the form $+\cdots +-\cdots -$. 
The signature rule says that $f_i$ (resp.\ $e_i$) acts 
on the element contributing the rightmost $+$ (resp.\ leftmost $-$) in this sequence, if it exists, and by zero otherwise. 
\end{corollary}

Considering such tensor products is very useful, since we can embed any crystal $B(\lambda) $ into a tensor product of simpler crystals $ B(\lambda^1) \otimes \cdots \otimes B(\lambda^n) $ where $ \lambda = \lambda^1 + \cdots + \lambda^n $ is a composition. This procedure can be used to define the crystal structure on semistandard Young tableau in type $A$ (see \cite{HK02}) and we will use it below for reverse plane partitions.

\subsection{Minuscule representations}\label{sec:minusculereps}

Let $J\subseteq I$ and let $W_J\le W$ denote the subgroup generated by $\{s_j : j\in J\}$.

\begin{proposition}
  (\cite[Corollary~2.4.5]{BB05})
  Every coset $wW_J$ has unique element of minimal length.
\end{proposition}

Denote by $w^J$ the element of minimal length in $wW_J$ and by $W^J$ the set of \textbf{minimal length representatives} for $\left\{wW_J : w\in W\right\}$. Note that $ W^J = \left\{ v \in W : v \leq_L w_0^J \right\}$.  

Recall that an irreducible $\fg$-representation $ V $ is said to be \new{minuscule} if $W$ acts transitively on the set of weights of $V$.  We say that $ \lambda \in P_+ $ is \new{minuscule} if $ V(\lambda) $ is a minuscule representation.  Equivalently, $ \lambda $ is minuscule if and only if $ \lambda - \alpha \notin P_+ $ for any root $ \alpha $.

\begin{remark}
	If $ \fg $ is simple, then every minuscule weight $ \lambda $ is fundamental; that is, there exists $ p \in I $ such that $ J  = I\setminus\{p\}$.  However, for semisimple $ \fg$, this is no longer the case.  For example, if we take $ \fg = \mathfrak{sl}_2 \oplus \mathfrak{sl}_2$, then the sum of the two fundamental weights is minuscule. 
\end{remark}

Let 
$ \lambda = \sum_{j \notin J} \omega_j $ where $\omega_j$ denotes the $j$th fundamental weight.
We have the following theorem, also due to Stembridge, connecting minuscule representations and dominant minuscule elements of the Weyl group.
\begin{theorem}(\cite[Theorem~6.1(b) and Theorem~7.1]{St96})
    Let $ J$ and $\lambda $ be as above.  Then $\lambda$ is minuscule if and only if every element of $W^J$ is fully commutative.  Moreover, in this case, every element of $W^J$ is dominant $\lambda$-minuscule, and the left weak order and strong order on $W^J$ coincide. 
\end{theorem}

Now assume that $ \lambda $ is minuscule. All the weight spaces $V(\lambda)_\mu$ are $1$-dimensional, and the weights lie in the $W$-orbit of $\lambda$.  As this orbit is in bijection with $ W / W_J = W^J$, the following result is immediate. 

\begin{proposition} \label{prop:WJcrystal}
The set 
$W^J$ is a model for the crystal $B(\lambda)$ with weight map given by $ wt(v) = v \lambda $ and following raising and lowering operators.
\begin{gather*}
  f_i(v)=
  \begin{cases}
    s_i v &\text{ if } s_i v > v  \text{ and } s_i v \in W^J \\
    0 &\text{ otherwise,}
  \end{cases}
  \qquad
  e_i(v)=
  \begin{cases}
    s_i v &\text{ if } s_i v < v\\
    0 &\text{ otherwise.}
  \end{cases}
\end{gather*}  
\end{proposition}

\begin{remark} \label{rem:orderideals}
From \Cref{thm:JH(w)weakorder}, we can view this as a crystal structure on the set of order ideals $J(H(w_0^J))$. 
In the language of bead configurations, which were introduced in \cref{sec:domminheaps}, applying the lowering operator $f_i$ (resp.\ the raising operator $e_i$) to a bead configuration adds a bead to runner $i$ (resp.\ removes a bead from runner $i$) if the result defines an order ideal in the heap $H(w_0^J)$. 
\end{remark}

\subsection{Minuscule Demazure crystals}
We will now generalize \Cref{prop:WJcrystal} to certain Demazure crystals.  Let $ \lambda \in P_+$ and $ w \in W $. We say that $ V_w(\lambda) $ is a \new{minuscule Demazure module}, if all weights of $ V_w(\lambda) $ lie in the Weyl orbit of $ \lambda $.

In this case, the corresponding crystal $ B_w(\lambda) $ is called a \new{minuscule Demazure crystal}.  If $ \lambda $ is minuscule, then it is easy to see that $ B_w(\lambda) $ is minuscule for all $ w \in W $.  However, we can also get minuscule Demazure crystals from dominant minuscule Weyl group elements.

\begin{proposition} \label{prop:dominant-minuscule-demazure}
  Let $ \lambda \in P_+$, $ w \in W $ and let $J= \left\{ j : s_j \lambda = \lambda \right\}$.
	
\begin{enumerate}
	\item	\label{it:if} If $ w$ is $\lambda$-minuscule, then  $B_w(\lambda)$ is a minuscule Demazure crystal.  Moreover, the set $[e,w]^J=\{v\in W^J : v\leq_L w\}$ is a model for the crystal $ B_w(\lambda) $ with weight map given by $ wt(v) = v \lambda $ and the following raising and lowering operators:
	\begin{gather*}
		f_i(v)=
		\begin{cases}
			s_i v &\text{ if } s_i v > v  \text{ and } s_i v \le_L w \\
			0 &\text{ otherwise,}
		\end{cases}
		\qquad
		e_i(v)=
		\begin{cases}
			s_i v &\text{ if } s_i v < v\\
			0 &\text{ otherwise.}
		\end{cases}
	\end{gather*}  

  \item	\label{it:only} Conversely, suppose $ B_w(\lambda) $ is a minuscule Demazure crystal and $ w \in W^J $.  Then $ w $ is $\lambda$-minuscule.
\end{enumerate}
\end{proposition}

As in \Cref{rem:orderideals}, if $ w $ is $ \lambda$-minuscule, this gives a model for $ B_w(\lambda) $ using the set of order ideals $J(H(w))$. 

\begin{proof}
  For 1, we proceed by induction on the length $\ell$ of $w $. Choose $ i $ such that $ s_i w < w $.  Then $ s_i w $ is also $\lambda$-minuscule.
  
  Let $ b \in B_w(\lambda) $.  Then by the definition of Demazure crystal, we can write $ b = f_i^m b' $ for some $b' \in B_{s_iw}(\lambda) $. 
  
  By induction, $ B_{s_iw}(\lambda) $ is a minuscule Demazure crystal, and so $ wt(b') = v\lambda $ for some $ v \in W^J $ with $ v \le_L s_i w $.  We claim that $ \langle \alpha_i^\vee, v \lambda \rangle \le 1 $.
  
  Using $ v \le_L s_i w $, we can form a reduced word $ s_i w = s_{j_2} \dots s_{j_r} s_{j_{r+1}} \dots s_{j_m} $ where $ v = s_{j_{r+1}} \dots s_{j_m} $. Thus $s_{j_1} \dots s_{j_m}$ is a reduced word for $ w$, where we declare $ j_1 = i$.  Choose $ p $ maximal such that $ p \le r $ and $ j_p = i $ (such a $ p $ exists since $ p = 1 $ satisfies these two conditions).  Then since $ w$ is $ \lambda$-minuscule, we conclude that
  $$ \langle \alpha_i^\vee, v \lambda - \alpha_{j_{p+1}} - \cdots - \alpha_{j_r} \rangle = 1 $$
  However, since $ j_q \ne i $ for $ q = p+1, \dots, r $, we see that $ \langle \alpha_i^\vee, \alpha_{j_q} \rangle \le 0 $ for $ q = p+1, \dots, r $.  Hence \begin{equation} \label{lesseq1}
  	\langle \alpha_i^\vee, v\lambda \rangle = 1 + \langle \alpha_i^\vee, \alpha_{j_{p+1}} \rangle + \dots + \langle \alpha_i^\vee, \alpha_{j_r} \rangle \le 1 
  \end{equation}
  
  Since $B_w(\lambda)$ is a subcrystal of the normal crystal $B(\lambda)$, $ m \le \phi_i(b^\prime) \le  \langle \alpha_i^\vee, v\lambda \rangle$ (unless this number is negative).  Hence either $ m = 0$ or $ m = 1$. 
  
  If $ m = 0 $, then $ b = b'$, and $ wt(b) = v \lambda $, and $ v \le_L w $ as $ v \le_L s_i w \le_L w$. 
  
  On the other hand, if $m = 1 $, then $\langle \alpha_i^\vee, v\lambda \rangle = 1$, $ b = f_i b' $ and $ wt(b) = s_i v \lambda$.  Moreover in this case, we have equality in (\ref{lesseq1}) above and so $ \langle \alpha_i^\vee, \alpha_{j_q} \rangle = 0 $ and hence $ s_i $ commutes with $ s_{j_q} $, for $ q = p+1, \dots, r $.  Thus we conclude that $$ w = s_{j_1} \dots s_{j_{p-1}} s_{j_{p+1}} \dots s_{j_r} s_i s_{j_{r+1}} \dots s_{j_m} $$
and hence $ s_i v \le_L w $.

	For 2, we also proceed by induction on $ \ell $.  Choose $ i $ such that $ s_i w < w $.  Then $ B_{s_iw}(\lambda) $ is also a minuscule Demazure 
  crystal, since $ B_{s_iw}(\lambda) \subset B_w(\lambda) $.  Moreover, we also have $ s_i w \in W^J $.  Thus by induction $ s_i w $ is $ \lambda$-minuscule.
	
	To show that $ w $ is $\lambda$-minuscule, it suffices to show that $ \langle \alpha_i^\vee, s_i w \lambda \rangle = 1 $.  To see this, note that $ \langle \alpha_i^\vee, s_i w \lambda \rangle > 0 $ (since otherwise $ s_iw\lambda = w\lambda $ contradicting $ w \in W^J$) and $ \langle \alpha_i^\vee, s_i w \lambda \rangle < 2 $, because $ f_i b_{s_i w\lambda} $ must have weight $ w \lambda $, as $ B_w(\lambda) $ is minuscule.
\end{proof}
%

%
%

The following result is essentially Remark 2.4 from \cite{St01}.
\begin{proposition}\label{thm:witness}
	Let $w\in W$ be dominant minuscule.
  The set of possible dominant witnesses $\lambda$ for $w$ is 
  \[
  \left\{
    \lambda_{min}+\eta : \lambda_{min}=\displaystyle\sum_{ws_i<w} \omega_i \text{ and }\eta\in\left(\displaystyle\bigcap_{j\in I'}(\alpha_j^\vee)^\perp\right) \cap P_+
  \right\}  
  \]
  with $I'$ denoting the set of simple reflections appearing in any reduced word for $w$. 
\end{proposition}

\begin{proof}
	Let $ \lambda $ be a dominant witness for $ w $.  Fix a reduced word $\bw = (s_{j_1},\dots,s_{j_\ell})$.

	First, we will prove that $\lambda $ has the desired form.  To do so, we need to show that $\langle  \alpha_i^\vee, \lambda \rangle = 1 $ if $ ws_i < w $ and that $ \langle  \alpha_i^\vee, \lambda \rangle = 0 $ if $ s_i $ occurs in the word for $ w $, but $ w s_i \not < w $.
	
	Let $ i \in I$ such that $ s_i $ occurs in the word for $ w $.  Choose $ 1 \le k \le l $ such that $ j_k = i $ but $ j_q \ne i $ for $ q > k $.
	
	Suppose that $ ws_i < w $.  Then $s_{j_q} $ commutes with $ s_i $ for all $ q > k $ and so 
	$$ \langle \alpha^\vee_i, \lambda \rangle = \langle \alpha^\vee_{j_k}, \lambda - \alpha_{j_\ell} - \cdots - \alpha_{j_{k+1}} \rangle = 1 $$
	as desired.
	
	Suppose that $ ws_i \not < w $.  Then at least one $ s_{j_q} $ must not commute with $ s_i $ for $ q > k $. By \Cref{dommin-condition}, there is exactly one such $ q $.  Thus,
	$$
	1 = \langle \alpha^\vee_{j_k}, \lambda - \alpha_{j_\ell} - \cdots - \alpha_{j_{k+1}} \rangle = \langle \alpha^\vee_i, \lambda \rangle - \langle \alpha^\vee_i, \alpha_{j_q} \rangle
	$$
	Since we are in the simply-laced situation, $ \langle \alpha^\vee_i, \alpha_{j_q}\rangle = -1 $ and so $ \langle \alpha^\vee_i, \lambda \rangle  = 0 $ as desired.
	
	Thus, we have proven that every witness has the desired form.
	
	For the reverse direction, since $ w $ is dominant minuscule, it has some dominant witness $ \lambda $.  By above, we can write $ \lambda = \lambda_{min}+\eta $ for $ \eta \in \displaystyle\bigcap_{i\in I'}(\alpha_i^\vee)^\perp\cap P_+$.  Then $s_{j_k} \cdots s_{j_\ell} \lambda = s_{j_k} \cdots s_{j_\ell} \lambda_{min} + \eta $.  Since $ \lambda $ is a witness, we have
	$$
	s_{j_k} \cdots s_{j_\ell} \lambda = \lambda - \alpha_{j_\ell} - \cdots - \alpha_{j_k} 
    $$
    whence
    $$
        s_{j_k} \cdots s_{j_\ell} \lambda_{min} = \lambda_{min} - \alpha_{j_\ell} - \cdots - \alpha_{j_k} 
    $$ 
    and $ \lambda_{min} $ is also a witness. It follows that $\lambda_{min}+\eta $ is a witness for any $ \eta \in \displaystyle\bigcap_{i\in I'}(\alpha_i^\vee)^\perp\cap P_+$.
\end{proof}

\begin{corollary}\label{thm:rightmaximal}
	Let $w$ be dominant minuscule and assume that a reduced word for $w$ contains each simple reflection at least once. Then $ w $ has a unique witness $\lambda$.  
\end{corollary}

\begin{example}
  The unique witness which appears in \Cref{thm:rightmaximal} does not need to be minuscule.  For example, consider type $ D_4$ and $ w = s_1s_3s_4 s_2 $.  Then $w $ is dominant minuscule with unique witness $\lambda = \omega_2$, which is not minusucle.  On the crystal level this means that not every minuscule Demazure crystal can be realized as the Demazure subcrystal of a minuscule $\fg$-crystal.
\end{example}

The following lemma will be useful for us later.

\begin{lemma}\label{thm:demazuretensor}
  Let $\lambda, \mu\in P_+$ and $w\in W$ (note that here we do not require that $w$ is dominant minuscule). Then $B_w(\lambda+\mu)$ can be identified with a subcrystal of $B_w(\lambda)\otimes B_w(\mu)$.
\end{lemma}
\begin{proof}
  We have the map
  \[
    \iota: B(\lambda+\mu)\hookrightarrow B(\lambda)\otimes B(\mu)
  \]
  by considering highest weights. By definition of Demazure crystals, $B_w(\lambda+\mu)\subseteq B(\lambda+\mu)$. The crystal $B_w(\lambda)\otimes B_w(\mu)$ is a subcrystal of $B(\lambda)\otimes B(\mu)$. We want to show that $\iota (B_w(\lambda+\mu))\subseteq B_w(\lambda)\otimes B_w(\mu)$.

  Let $\mathbf{w}=(s_{i_1},s_{i_2},\ldots , s_{i_\ell})$ be a reduced word for $w$. Let $b\in B_w(\lambda+\mu)$. Then by definition of Demazure crystals, there exist $p_k\in\bN$ such that $b=f_{i_1}^{p_1}f_{i_2}^{p_2}\cdots f_{i_\ell}^{p_\ell} b_{\lambda+\mu}$.  Then,
  \begin{align*}
    \iota(b) &= f_{i_1}^{p_1}f_{i_2}^{p_2}\cdots f_{i_\ell}^{p_\ell} \iota(b_{\lambda+\mu}) && \text{since $\iota$ is a crystal morphism}\\
             &= f_{i_1}^{p_1}f_{i_2}^{p_2}\cdots f_{i_\ell}^{p_\ell} (b_\lambda \otimes b_\mu) & \\
             &= f_{i_1}^{p_1}f_{i_2}^{p_2}\cdots f_{i_{\ell-1}}^{p_{\ell-1}} (f_{i_\ell}^{q_\ell} b_\lambda \otimes f_{i_\ell}^{p_\ell-q_\ell} b_\mu) & & \text{for some }0\leq q_\ell \leq p_\ell\\
             &\,\,\,\vdots \\
             &= (f_{i_1}^{q_1}f_{i_2}^{q_2}\cdots f_{i_\ell}^{q_\ell} b_\lambda) \otimes (f_{i_1}^{p_1-q_1}f_{i_2}^{p_2-q_2}\cdots f_{i_\ell}^{p_\ell-q_\ell} b_\mu  )
  \end{align*}
  which shows that $\iota(b)\in B_w(\lambda)\otimes B_w(\mu)$.
\end{proof}

\begin{remark}
	It is tempting to guess that $ B_w(\lambda + \mu) = B(\lambda + \mu) \cap \left( B_w(\lambda) \otimes B_w(\mu) \right)$ where we embed both $ B(\lambda + \mu) $ and $ B_w(\lambda) \otimes B_w(\mu) $ as subsets of $ B(\lambda) $.  In general, this is not true, as can be seen for $ \fg = \mathfrak{sl}_3$,  $ \lambda = \omega_1$, $\mu = \omega_2$ and $ w = s_1 s_2 $; in this case $ B_w(\lambda + \mu)$ contains 5 elements, while  $B(\lambda + \mu) \cap \left( B_w(\lambda) \otimes B_w(\mu) \right)$ contains 6 elements.
	
	Recently the first author and collaborators \cite{AnneNew} have shown that the above guess is almost always true.  They proved that if $ \lambda, \mu $ satisfy the condition that whenever $ \alpha_i^\vee(\mu) = 0 $, then $ \alpha_i^\vee(\lambda) = 0 $, then for any $ w $, $$ B_w(\lambda + \mu) = B(\lambda + \mu) \cap \left( B_w(\lambda) \otimes B_w(\mu) \right) $$
	In particular for every positive integer $n $, we have $ B_w(n\lambda) = B(n\lambda) \cap B_w(\lambda)^{\otimes n} $. 
\end{remark}

\subsection{Reverse plane partitions}
\label{ss:rpps}
Let $ \lambda \in P_+$ and $ w\in W $ be $\lambda$-minuscule.
In this section, we will describe a model for the Demazure crystal $B_w(n\lambda)$. 

\begin{definition}
  A \textbf{reverse plane partition of shape $H(w)$} with height $n \in \bN$, is an element of the set
  \[
    \Rpp(w,n)=\left\{ \Phi: H(w) \to \{0,1,\ldots n\} \;|\; \Phi \text{ is order-reversing} \right\}
  \]
  where $\{0,1,\ldots ,n\}$ is totally ordered.
\end{definition}
\begin{remark}\label{eg:indicator}
  By identifying subsets of $H(w)$ and their indicator functions we see that 
  \[
    \Rpp(w,1) = J(H(w)).
  \]
\end{remark}
Define a map $ \Rpp(w,n) \rightarrow J(H(w))^n $ by
\begin{equation} \label{eq:gravsort}
\Phi \mapsto (\phi_1, \dots, \phi_n) \qquad \phi_k = \Phi^{-1}(\{n-k+1, \dots, n\}).
\end{equation}

The following simple observation will be very useful for us. 
\begin{lemma}\label{lem:gravsort}
	This map is injective with image $ \{ (\phi_1, \dots, \phi_n) : \phi_k\subseteq \phi_{k+1} \text{ for $ k \ge 1 $}  \}$.
\end{lemma}
A sequence $\phi_1, \dots, \phi_n $ satisfying $ \phi_k\subseteq \phi_{k+1} $ for $ k \ge 1 $, will be called an \new{increasing chain} of order ideals. Hence we can identify an RPP with its corresponding increasing chain of order ideals.

A repeated application of \Cref{thm:demazuretensor} shows that $B_w(n\lambda)$ is a subcrystal of $(B_w(\lambda))^{\otimes n}$ so we may identify $B_w(n\lambda)$ with its image in $(B_w(\lambda))^{\otimes n}$. Increasing chains of order ideals also appear when we consider $B_w(n\lambda)$.  Here and below, we use \Cref{prop:dominant-minuscule-demazure} to identify $ B_w(\lambda) $ and $ J(H(w))$. The main result of this section is the following result.

\begin{theorem} \label{thm:gravsort}
The map $B_w(n\lambda)\hookrightarrow \left(B_w(\lambda)\right)^{\otimes n} $ has the same image as in \Cref{lem:gravsort} and thus we have a bijection $ \Rpp(w,n) \cong B_w(n\lambda) $.
\end{theorem}
This theorem endows $ \Rpp(w,n) $ with a crystal structure.  In this crystal structure, the weight map is given by
\begin{equation}
wt(\Phi) = \lambda - \sum_{x \in H(w)} \Phi(x) \alpha_{\pi(x)} \label{eq:weight}
\end{equation}

\begin{remark} \label{re:LS}
	Given an increasing chain $\phi_1, \dots, \phi_n $ of order ideals, we can use \Cref{thm:JH(w)weakorder} to produce an increasing chain $ v_1, \dots, v_n $ in the Bruhat order on $ \{ v \in W : v \le_L w \} $.  This chain of Weyl group elements defines a Lakshmibai-Sesahdri path in the sense of \cite[Section 2]{L94}.  Using this bijection, it is possible to deduce \Cref{thm:gravsort} from the main results of \cite{L94}.  Alternatively, when $ \lambda $ is minuscule, \Cref{thm:gravsort} is equivalent to \cite[Prop 7.29]{Scrim}.  However, we prefer to give a self-contained proof in the language of reverse plane partitions.  
\end{remark}

The crystal structure on $\left(B_w(\lambda)\right)^{\otimes n}$ can be computed by the signature rule (\Cref{thm:signrule}). Note that in our situation every tensor factor is a minuscule Demazure crystal so every factor will contribute at most one sign to the sequence.

\begin{example}
\label{example:heap}
Let $\fg=\mathfrak{sl_4}$ and $\mathbf{w}=(s_2,s_3,s_1,s_2)$. The glass bead game of this heap is depicted in \Cref{fig:smallerheap}. 

Since $w = w_0^{\{1,3\}}$, we know that $J(H(w))$ is a model for the crystal $B(\omega_2)$.
Consider the element 
\[
b = \begin{psmallmatrix}
	& 0 & \\
	0 & & 0 \\
	& 1 &
\end{psmallmatrix}\otimes
\begin{psmallmatrix}
	& 0 & \\
	0 & & 1 \\
	& 1 &
\end{psmallmatrix}\otimes
\begin{psmallmatrix}
	& 0 & \\
	1 & & 1 \\
	& 1 &
\end{psmallmatrix}\otimes
\begin{psmallmatrix}
	& 1 & \\
	1 & & 1 \\
	& 1 &
\end{psmallmatrix} 
\]
in $B(\omega_2)^{\otimes 4}$. To apply the operators $f_2, e_2$ we compute the sign pattern  $ - \phantom{+} + -$.  

After cancelling all $ -+ $ pairs, we are left with just a single $ - $ in the fourth position.  So we see that $ f_2(b) = 0 $, while $ e_2 $ acts on the fourth tensor factor to yield
\[
e_2(b) = \begin{psmallmatrix}
	& 0 & \\
	0 & & 0 \\
	& 1 &
\end{psmallmatrix}\otimes
\begin{psmallmatrix}
	& 0 & \\
	0 & & 1 \\
	& 1 &
\end{psmallmatrix}\otimes
\begin{psmallmatrix}
	& 0 & \\
	1 & & 1 \\
	& 1 &
\end{psmallmatrix}\otimes
\begin{psmallmatrix}
	& 0 & \\
	1 & & 1 \\
	& 1 &
\end{psmallmatrix} \,.
\]
\begin{figure}[ht]
	\centering
    \begin{tikzpicture}
      \foreach \x in {0,1,2} 
        \draw [line width=1mm] (\x, 0) -- (\x,4);
      \foreach \x in {0,1} 
        \shade[shading=ball, ball color=white] (\x+1,\x+1) circle (0.7);
      \foreach \x in {0,1} 
        \shade[shading=ball, ball color=white] (\x,\x+2) circle (0.7);
      \draw [line width=0.5mm] (0,0) -- (2,0);
      \foreach \x/\xtext in {0/1,1/2,2/3} 
        \draw (\x cm,0 pt) node[fill=white,inner sep=2pt,draw] {${\xtext}$};
    \end{tikzpicture}
	\caption{The heap $H((s_2,s_3,s_1,s_2))$ in type $A_3$}
	\label{fig:smallerheap}
\end{figure}
\end{example}
\begin{lemma} \label{lem:preserveincreasing}
    Suppose that $ b = \phi_1 \otimes \cdots \otimes \phi_n \in B_w(\lambda)^{\otimes n} $ is an increasing chain of order ideals.  Then for any $ i$, $ e_i(b) $ and $ f_i(b) $ are also increasing (if they are non-zero).
\end{lemma}

\begin{proof}

Let us write $ e_i(b) = \phi_1' \otimes \cdots \otimes \phi'_n$.  By the definition of the crystal structure on tensor products, there exists $ j $ such that $ \phi'_j = e_i(\phi_j) $ and $\phi'_k = \phi_k$ for all $k \ne j $.

Since $ \phi'_j $ is obtained from $ \phi_j $ by removing a bead on runner $ i$ and all the rest are unchanged, the condition $ \phi'_k \subseteq \phi'_{k+1}$ definitely holds at all $ k \ne j-1 $.  So it suffices to check that $ \phi'_{j-1} \subseteq \phi'_j $.  The only way that this could be violated is if the bead $x $ on runner $ i $ removed by applying $ e_i $ to $ \phi_j $ is an element of $ \phi_{j-1} $.  

If $ x \in \phi_{j-1} $, then it can be removed, since all the beads in $\phi_{j-1} $ lie below $ x $ in $ H(w) $ (as this holds in $ \phi_j $).  However, if a bead can be removed from runner $i$ in $\phi_{j-1}$ then $ \phi_{j-1} $  contributes a $- $ to the sign pattern before $\phi_{j-1}$, and the signature rule says $ e_i $ acts on the leftmost $-$ so not on $ \phi_j $.  
Therefore, the bead $ x $ cannot be in $ \phi_{j-1} $, so we conclude that $  \phi'_{j-1} \subseteq \phi'_j $.

The argument for $ f_i $ is similar. Let us write $ f_i(b) = \phi_1' \otimes \cdots \otimes \phi'_n$. Again, there exists $ j $ such that $ \phi'_j = f_i(\phi_j) $ and $\phi'_k = \phi_k$ for all $k \ne j $.

To check the increasing condition, we need to show that  $ \phi'_j \subseteq \phi'_{j+1}$.  Suppose that $ x $ is the bead that is added to $ \phi_j $ to form $ \phi'_j $. 

By the signature rule, no beads can be added to $ \phi_{j+1} $ on the $ i$th runner.  Since $ x $ can be added for $ \phi_j $, and $ \phi_j \subseteq \phi_{j+1}$, this means that we must have $ x \in \phi_{j+1} $.  Hence we conclude that $  \phi'_j \subseteq \phi'_{j+1} $.
\end{proof}

\begin{proposition}\label{thm:gravitysort}
  Let $\mathbf{\Phi}=\phi_1\otimes \cdots \otimes \phi_n$ be in the image of the embedding $B_w(n\lambda)\hookrightarrow \left(B_w(\lambda)\right)^{\otimes n}$. Then $ \phi_1, \dots, \phi_n $ is an increasing chain of order ideals.
\end{proposition}

\begin{proof}
  By definition, every element of $ B_w(n\lambda) $ can be obtained by applying lowering operators to the highest weight element $b_{n\lambda}$ which corresponds to $\phi_k=\emptyset$ (thought of as an order ideal of $H(w)$) for all $k$.  Thus, the result follows from \Cref{lem:preserveincreasing}, by induction using the lowering operators.
\end{proof}

Before proving the converse of \Cref{thm:gravitysort} we need a Lemma.

\begin{lemma}\label{thm:zeroentry}
  Let $\Phi\in RPP(w,n)$ and let $s_i$ be such that $s_i w<w$. Let $x$ be the element of $H(w)$ corresponding to this (first) occurrence of $s_i$ and let
  \[
    \Psi=e_i^{\varepsilon_i(\Phi)}(\Phi).
  \]
  Then the entry $\Psi(x)$ is zero.
\end{lemma}

\begin{proof}
  Since $RPP(w,n)\subseteq B_w(\lambda)^{\otimes n}$ is upper semi-normal, $e_i^{\varepsilon_i(\Phi)}(\Phi)\neq 0$, and $\varepsilon_i(\Psi)=0$. Let $\psi_1\otimes \cdots \otimes \psi_n$ be the increasing chain of order ideals corresponding to $\Psi$. If $\Psi(x)\neq 0$, then there are tensor factors $\psi_k$ for which $x\in \psi_k$. These must occur at the rightmost positions in the tensor product since $RPP(w,n)$ consists of increasing chains of order ideals.

  Consider the $e_i/f_i$ sign pattern of $\Psi$. The beads corresponding to $x$ are all removable because $x$ is a maximal element, so they contribute $-$ signs to the sign pattern. Note that $x$ is a maximal element of $H(w)$ so these $-$ signs can not be followed by $+$ signs, so they can never be cancelled, therefore they contribute to $\varepsilon_i(\Psi)$, but this was shown to be zero, so we have a contradiction.
\end{proof}

\begin{proposition}\label{thm:conversegravsort}
  Let $\phi_1,\ldots ,\phi_n$ be an increasing chain of order ideals. Then the corresponding element $\Phi=\phi_1\otimes\cdots \otimes \phi_n$ is in the image of the embedding $B_w(n\lambda)\hookrightarrow (B_w(\lambda))^{\otimes n}$.
\end{proposition}

\begin{proof}
  We use induction on $l(w)$. Let $\Phi\in RPP(w,n)$. By \Cref{thm:zeroentry}, we may think of $\Psi=e_i^{\varepsilon_i(\Phi)}(\Phi)$ as an element of $RPP(s_i w, n)$. As $l(s_iw)<l(w)$, we have $RPP(s_iw,n)\subseteq B_{s_iw}(n\lambda)$. Since $\Psi$ is obtained from $\Phi$ by applying only $e_i$-s, we have $\Phi\in B_w(n\lambda)$.
\end{proof}

We are now in a position to complete the proof of \Cref{thm:gravsort}.
\begin{proof}[Proof of \Cref{thm:gravsort}]
  The two containments for the equality
  \[
    RPP(w,n)=B_w(n\lambda)
  \]
  (where both are considered as subcrystals of $B_w(\lambda)^{\otimes n}$) are \Cref{thm:gravitysort} and \Cref{thm:conversegravsort}.
\end{proof}

\subsection{Tableaux, Gelfand--Tsetlin patterns, and RPPs}
\label{se:tableaux}

In type $A$, RPPs come from Gelfand--Tsetlin (or GT) patterns of rectangular semistandard Young tableaux.
Take $ \fg = \mathfrak{sl}_m $ and let $ \lambda = (\lambda_1\ge \cdots\ge \lambda_m = 0) $ be a dominant weight.  Consider the set of semistandard Young tableaux of shape $ \lambda $ with labels $\{1, \dots, m \} $ which we denote by $ SSYT(\lambda) $.  The GT pattern of such a tableau $\tau$ is a triangular array $ gt(\tau) = (\lambda^{(1)}, \lambda^{(2)}, \dots, \lambda^{(m)})$ where $\lambda^{(i)} = \lambda^{(i)}(\tau)$ is the shape of
the subtableau comprised of boxes labelled \(\{1,\dots,i\}\). The information of this array completely determines the tableau. 

For example,
\[
  gt\left(\young(113,224)\right) 
  = \begin{matrix}
    & & & 2 \\
    & & 2 & & 2 \\
    & 3 & & 2 & & 0\\
    3 & & 3 & & 0 & & 0
  \end{matrix} 
\]
When $\lambda = (n^{p}) = (n,\dots,n, 0, \dots, 0)$ for some $ 1 \le p \le m-1 $ and $ n \in \bN $, then $ gt(\tau) $ has the form
\[
gt(\tau) = \begin{matrix}
  & & & & & & \lambda^{(1)}_1 & \\
  & & & & & \Ddots & & \ddots & \\
  & & & & \lambda^{(m-p)}_1 & & & & \ddots &  \\
  & & & n & & \ddots & & & & \lambda^{(p)}_p &  \\
  & & \Ddots & & \ddots  & & \ddots & & \Ddots &  & 0 & \\
  & n & &  & & n & & \lambda^{(m-1)}_p &  & 0 &  & 0 & \\
  n & & \cdots & & \cdots & & n & & 0 & & \cdots  & & 0 
\end{matrix}
\]
Hence, the information necessary to reconstruct $ \tau $ is contained in a $p\times (m-p)$ rectangular array.

Consider the rectangular array $\Phi(\tau)$ obtained from this one by reflecting in a vertical axis, and then rotating 90 degrees clockwise. 
\[
\Phi(\tau) = 
\begin{matrix}
  & & \lambda_p^{(p)} \\
  & \Ddots & & \ddots \\
  \lambda_p^{(m-1)} & & & & \ddots \\
  & \ddots & & & & \lambda_1^{(1)} \\
  & & \ddots & & \Ddots \\ 
  & & & \lambda_1^{(m-p)} 
\end{matrix}  
\]
Since GT patterns satisfy \(\lambda^{(i+1)}_{k+1}\le \lambda^{(i)}_k\le \lambda^{(i+1)}_k\) for all \(1\le k\le i\le m\), we can regard $\Phi(\tau)$ as a RPP of height $ n$.
More precisely, it is a RPP of shape $H(w_0^J)$ where $J = \{1,\dots,m-1\}\setminus \{p\}$. In 1-line notation $w_0^J \in S_m$ is the permutation $m-p+1\, m-p + 2\dots m \, 1\, 2\dots m-p$. As a product of generators we can take $w_0^J = (s_{m-p}\cdots s_{m-2}s_{m-1})\cdots (s_2 \cdots s_p s_{p+1})(s_1\cdots s_{p-1} s_p)$. 

From the above discussion, we conclude the following.
\begin{theorem}\label{thm:gtrpp}
  Let $n,p,J$ be as above. 
  The map $\tau\mapsto \Phi(\tau)$ defines a bijection $ SSYT(n^p) \cong \Rpp(w_0^J,n) $.
\end{theorem}

\begin{remark} \label{rem:SI}
	We warn the reader that this is not a crystal isomorphism with respect to the usual crystal structure on $ SSYT(n^p) $, as this bijection switches the roles of $ e_i $ with $ f_{m-i}$ and acts by $ w_0 $ on the weight.  In other words, this bijection implements the Sch\"utzenberger involution.
\end{remark}

\subsection{Toggle groups and Cactus groups}
\label{subsec:togglecactus}

In this section we will describe some additional motivation for thinking about RPPs, and outline some directions for future research related to toggles.

In \cite{BK72}, Bender and Knuth introduced involutions $t_i$ on semistandard Young Tableaux; they are now called \textbf{Bender--Knuth involutions}. They act by changing certain entries of a tableau from $i$ to $i+1$ and vice versa. More precisely, for each row of the tableau, declare an entry $i$ (or $i+1$) to be free if there is no $i+1$ (or $i$) in the same column. In each row, the free $i$'s and $i+1$'s occur in a single horizontal strip. Then $t_i$ acts be reversing the number of free $i$'s and the number of free $i+1$'s in each horizontal strip.

These involutions have some surprising applications. For example, in \cite{St02} Stembridge used them to give a quick proof of the Littlewood--Richardson rule. In \cite{BK95}, Berenstein and Kirillov considered the group generated by the Bender--Knuth involutions. 
They demonstrated that many combinatorial constructions on tableaux, for example, the Sch\"{u}tzenberger involution or promotion, can be described using elements of this group.

It is clear from the definition that the $t_i$'s are involutions and it is not difficult to see that $t_it_j=t_jt_i$ if $|i-j|>1$, but other relations in this group are more mysterious. Berenstein and Kirillov show that $(t_1t_2)^6=id$ and $(t_1q_i)^4=id$, where $q_i=t_1t_2t_1\cdots t_it_{i-1}\cdots t_1$ and $i\geq 3$, and they conjectured that these are the defining relations for this group (the \textbf{BK group}).

On the other hand, a closely related group, also generated by involutions, is the cactus group.
\begin{definition}
 Fix a semisimple Lie algebra $\fg$ with Dynkin diagram $I $.  For any connected subdiagram $ J \subseteq I$, let $ w_0^J $ be the longest element of the Weyl group $ W_J $ and let $ \theta_J : J \rightarrow J $ be defined by $ w_0^J(\alpha_j) = -\alpha_{\theta_J(j)}$.  The \textbf{cactus group} $ C_\fg $ is generated by elements $s_J$ indexed by connected subdiagrams of $J$ subject to the following relations:
\begin{enumerate}
\item $s_J^2=1\; \forall\; J\subseteq I$
\item $s_Js_{J^\prime}=s_{\theta_J(J^\prime)}s_J\; \forall\; J^\prime\subseteq J$
\item $s_Js_{J^\prime }=s_{J^\prime}s_J\; \forall\; J^\prime,J\subseteq I$ such that $J\cup J^\prime$ is not connected.
\end{enumerate}
There is a map $ C_\fg \rightarrow W $ given by $ s_J \mapsto w_0^J $.
\end{definition}

The cactus group was first studied in relation to crystals by Henriques and the third author \cite{HK}. In their constructions, they defined a map $\xi_B$ on a crystal $B$ of any Cartan type that generalizes the Sch\"{u}tzenberger involution in type $A$. In \cite{Hal}, Halacheva studied what happens when we start with a $\fg$-crystal $B$, then treat it as a $\fg_J$-crystal $B_J $ (where $J$ is a connected subdiagram), then apply $\xi_{B_{J}}$. In type $A$, Berenstein and Kirillov called this operation a ``partial Sch\"{u}tzenberger involution''. This can be considered as a bijection from $B$ (considered as a set) to itself. Halacheva showed that these elements $\xi_{B_{J}}$ satisfy the relations of the cactus group, thus the cactus group acts on the crystal $ B$.
 
In type $ A$, Berenstein and Kirillov described the partial Sch\"{u}tzenberger involutions in terms of the generators of the BK group.  Thus, there is a map from the cactus group to the BK group, and the cactus group action on $\mathfrak{sl}_m$-crystals is compatible with this map. In \cite{CGP}, Chmutov, Glick and Pylyavksyy independently showed that there is a map from the cactus group to the BK group.

Since the cactus group is defined for any Cartan type, this raises the question of whether there is an analogue of the Bender--Knuth involutions for other types.

Let $P$ be any poset.  The Bender--Knuth involutions are examples of sequences of combinatorial toggles (some closely related and well-studied concepts are combinatorial rowmotion and promotion \cite{SW12}, \cite{CFdF95}). A \textbf{combinatorial toggle} is a map $t_x: J(P)\to J(P)$ given by
\[
  t_x(S)=
  \begin{cases}
    S\cup \{x\},&\text{ if }x\not\in S\text{ and }S\cup \{x\} \in J(P) \\
    S\setminus \{x\},&\text{ if }x\in S\text{ and }S\setminus \{x\} \in J(P) \\
    S &\text{ otherwise.}
  \end{cases}
\]
They are called combinatorial toggles since their effect is to ``toggle'' an element into or out of an order ideal. There is an extension of toggles to $RPP(w,n)$ given by the formula:
\[
  (t_x(\Phi))(x)=\max_{y<x}( \Phi(y) )+\min_{z> x} (\Phi(z) )-\Phi(x),
\]
where we define $\max_{y<x}(\Phi(y))=n$ if there is no $y$ such that $y<x$, and similarly $\min_{z>x}(\Phi(z))=0$ if there is no $z$ such that $z>x$. The other entries of the RPP are unchanged, that is, $(t_x(\Phi))(y)=\Phi(y)$ for $y\neq x$. 

Now, we specialize to the case where $ P = H(w) $ is a dominant minuscule heap. 
Note that if $ \pi(x) = \pi(y) $ then $t_x $ and $ t_y$ commute (here $ \pi : H(w) \rightarrow I $ is the map described after \Cref{def:heap}).  

When we work with a rectangular dominant weight $ \lambda$ and interpret semistandard Young tableaux of shape $ \lambda $ as RPPs of shape $ H(w)$ as in \Cref{se:tableaux}, the BK involutions can be reinterpreted using toggles on $ \Rpp(w,n)$.
\begin{align}
  t_i&=\prod_{\pi(x)=i}t_x, \label{eqn:BKops}
\end{align}

This lets us define a \textbf{toggle group} for arbitrary types, as long as we fix a heap $H(w)$. To be precise, we define involutions $t_i : \Rpp(w,n) \rightarrow \Rpp(w,n) $, for $ i \in I$, by the formula in \cref{eqn:BKops}, and we define the toggle group $Tog(w) $ as the subgroup of $\prod_n S_{RPP(n,w)}$ generated by the $t_i$. Here $S_{RPP(n,w)}$ is the symmetric group permuting $RPP(n,w)$.

To be able to compare the actions of the cactus and toggle groups, we have to restrict our attention to $\fg$-crystals, since Demazure crystals do not admit an action of the cactus group in general. To make the exposition easier, for the rest of this section we will assume that $\fg$ is simple.

\begin{definition}\label{def:bidomindantminuscule}
  Let $w\in W$. We say that $w$ is \textbf{bidominant minuscule} if $w$ is minuscule with a dominant witness $\lambda$ such that $w(\lambda)$ is antidominant.
\end{definition}

In particular, if $w$ is bidominant minuscule, then $RPP(H(w),n)$ is a $\fg$-crystal, not just a Demazure crystal. The following result gives an alternative characterization of bidominant minuscule elements.

\begin{proposition}
  \label{thm:inversedominantminuscule} Let $w\in W$ and assume that a reduced word for $w$ uses all simple reflections of $W$. Then $w$ is bidominant minuscule if and only if $w$ and $w^{-1}$ are both dominant minuscule.
\end{proposition}

\begin{proof}
  Assume that $w$ is bidominant minuscule. Let $(s_{i_1},\ldots , s_{i_\ell})$ be a reduced word for $w$, then $(s_{i_\ell},\ldots , s_{i_1})$ is a reduced word for $w^{-1}$. Let $\mu=-w_0(\lambda)=-w(\lambda)$. Since $w(\lambda)$ is antidominant by assumption, $\mu$ is dominant. We claim that $\mu$ is a witness for $w^{-1}$. To see this, consider
  \begin{align*}
    s_{i_k} \cdots s_{i_1} (\mu) &= - (s_{i_k} \cdots s_{i_1}) (s_{i_1} \cdots s_{i_k})(s_{i_{k+1}} \cdots s_{i_\ell})(\lambda) \\
                                 &= - (s_{i_{k+1}} \cdots s_{i_\ell})(\lambda) \\
                                 &= -(\lambda -\alpha_{i_\ell} - \ldots - \alpha_{i_{k+1}}) \\
                                 &= \mu - \alpha_{i_k}-\ldots - \alpha_{i_1}                                   
  \end{align*}
  so $w^{-1}$ is $\mu$-minuscule, and therefore dominant minuscule.

  For the other direction, assume that $w$ and $w^{-1}$ are both dominant minuscule with respective witnesses $\lambda$ and $\mu$. That is,
  \[
    s_{i_k}\cdots s_{i_\ell}(\lambda) = \lambda -\alpha_{i_\ell} - \ldots -\alpha_{i_k}
  \]
  and
  \[
    s_{i_k}\cdots s_{i_1}(\mu) = \mu - \alpha_{i_1}-\ldots - \alpha_{i_k}
  \]
  for all $k$ with $\lambda, \mu$ dominant.

  By a similar computation as above, we have
  \[
    s_{i_k}\cdots s_{i_1}(-w(\lambda)) = -w(\lambda) -\alpha_{i_1}-\ldots - \alpha_{i_k}
  \]
  for all $k$.  Therefore,
  \[
    s_{i_k}\cdots s_{i_1}(\mu+w(\lambda)) = \mu+w(\lambda)
  \]
  for all $k$. Since $w$ (and therefore $w^{-1}$) uses all simple reflections in $W$, this means that $\mu+w(\lambda)$ is fixed by all simple reflections, and therefore $\mu=-w(\lambda)$. In particular, $-w(\lambda)$ is dominant, therefore $w(\lambda)$ is antidominant and hence $w$ is bidominant minuscule.
\end{proof}

Now we can use \Cref{dommin-condition} to obtain a heap-theoretic consequence of $w$ being bidominant minuscule:

\begin{proposition}
  \label{thm:bidomlemma} If $\mathbf{w}=(s_{i_1}, \ldots , s_{i_\ell})$ is a reduced word for $w$ and $w$ is bidominant minuscule then the conditions of \Cref{dommin-condition} are satisfied, and, in addition, the first occurrence of each generator in $\mathbf{w}$ is preceded by at most one generator that does not commute with it.
\end{proposition}

Assume for the rest of this section that $ w $ is bidominant minuscule with dominant witness $ \lambda$, so that $ \Rpp(w,n) = B(n\lambda) $ is $\fg$-crystal.  To help with computations, it is useful to note the toggle group interacts nicely with the weight of an RPP.

\begin{lemma}
  For any $ \Phi \in \Rpp(w,n)$, we have $ wt(t_i(\Phi)) = s_i (wt(\Phi))$.
\end{lemma}

\begin{proof}
  We compute, using \Cref{eq:weight}
  \begin{align*}
    wt(t_i(\Phi)) &= n\lambda - \sum_{x\in H(w)}t_i(\Phi)(x)\alpha_{\pi(x)} \\
                  &= n\lambda - \sum_{\pi(x)\neq i}\Phi(x)\alpha_{\pi(x)} - \sum_{\pi(x)=i}\left( t_i(\Phi)\right)(x) \alpha_i \\
                  &= \left(n\lambda - \sum_{x \in H(w)}\Phi(x)\alpha_{\pi(x)}\right) - \sum_{\pi(x)=i} \left( \max_{y<x}(\Phi(y)) + \min_{z>x}(\Phi(z)) -2 \Phi(x) \right) \alpha_i \\
                  &= wt(\Phi) - \sum_{\substack{x\in H(w)\\ \pi(x)=i}} \left( \max_{y<x}(\Phi(y)) + \min_{z>x}(\Phi(z)) -2 \Phi(x) \right) \alpha_i
  \end{align*}
  and
  \begin{align*}
    s_i(wt(\Phi)) &= wt(\Phi) - \left\langle n\lambda - \sum_{x\in H(w)}\Phi(x)\alpha_{\pi(x)} , \alpha_i \right\rangle \alpha_i. \\
                  &= wt(\Phi) - \left( \langle n\lambda , \alpha_i \rangle - \left\langle \sum_{\substack{x\in H(w)\\ \pi(x)=i}} \Phi(x) \alpha_i,\alpha_i \right\rangle - \left\langle \sum_{\substack{y\in H(w) \\ \langle \pi(y),\alpha_i\rangle=-1}} \Phi(y)\alpha_{\pi(y)},\alpha_i\right\rangle \right)\alpha_i \\
                                                                                            &= wt(\Phi)- \left( \langle n\lambda , \alpha_i \rangle -2\sum_{\substack{x\in H(w)\\ \pi(x)=i}} \Phi(x) + \sum_{\substack{y\in H(w)\\\langle \pi(y),\alpha_i\rangle=-1}} \Phi(y)\right)\alpha_i
  \end{align*}



  So to get the desired equality, we have to show that
  \begin{equation}
    \sum_{\substack{x\in H(w)\\\pi(x)=i}} \max_{y<x}(\Phi(y)) + \min_{z>x}(\Phi(z)) = \langle n\lambda, \alpha_i \rangle +\sum_{\substack{y\in H(w)\\\langle \pi(y),\alpha_i\rangle=-1}}\Phi(y) \label{eq:weightsimplified}
  \end{equation}
  
  Notice that $\langle \pi(y),\alpha_i\rangle=-1$ if and only if the runner $\pi(y)$ is adjacent to runner $i$. In particular, the sum $\sum_{\langle \pi(y),\alpha_i\rangle=-1} \Phi(y)$ is the same as adding up all the entries of $\Phi$ on runners adjacent to $i$. 

  We have to show that each entry on the right-hand-side of \Cref{eq:weightsimplified} occurs on the left-hand side exactly once as a minimum or maximum for some $x$.

  Label the beads on runner $i$ top to bottom as $x_1,\ldots, x_m$ (so $x_1$ is the highest, and $x_k$ is the lowest bead on runner $i$). By \Cref{thm:lminuscule}, for $k=1,\ldots , m-1$, between $x_k$ and $x_{k+1}$ there are exactly two beads on runners adjacent to $i$. Let $y_k$ denote the one with larger $\Phi$-value and let $z_k$ denote the one with smaller $\Phi$-value. Notice that we have $\Phi(y_k)=\max_{y<x_{k+1}}\Phi(y)$ and $\Phi(z_k)=\min_{z>x_{k}}\Phi(z)$.

  There are possibly two more beads that are unaccounted for at this point, as by \Cref{thm:bidomlemma} there may be at most one bead on an adjacent runner above $x_1$ and at most one below $x_m$.
  If there is a bead on an adjacent runner above $x_1$ then this corresponds to the $\min_{z>x_1}(\Phi(z))$ term on the left-hand side. If there is no such bead, then $x_1$ is a highest element, and by definition $\min_{z>x_1}(\Phi(z))=0$.
  If there is a bead on an adjacent runner below $x_m$ then this corresponds to the $\max_{y<x_m}(\Phi(y))$  term on the left-hand side, and the lowest element of the heap is not on runner $i$, so $\lambda=\omega_j$ for some $j\neq i$, so $\langle n\lambda, \alpha_i \rangle=0$. If there is no such bead then $x_m$ is a lowest element, so $\lambda=n\omega_i$ and by definition $\max_{y<x_m}(\Phi(y))=n=\langle n\lambda , \alpha_i\rangle$. 

  So in all cases, \Cref{eq:weightsimplified} is satisfied and we are done.

\end{proof}

Now we are in a similar situation as in type $A$; for any $n \in \mathbb N$, we have a crystal $ B(n\lambda) = \Rpp(H(w), n) $, by \Cref{thm:gravsort}, with actions of two finitely generated groups, the cactus group $C_\fg $ (depending only on $ \fg$) and the toggle group $ Tog(w)$ (which depends on $ \lambda$, or equivalently $ w$).

\begin{conjecture}
  There is a surjective map $ C_\fg \rightarrow Tog(w)$ such that the action of the cactus group on $ B(n\lambda) = \Rpp(w,n) $ factors through this map.
\end{conjecture}

For example, in type $D_4$, for the crystal $B(n\omega_1)$, the toggles $t_3$ and $t_4$ act just like the cactus group elements $s_{\{3\}}, s_{\{4\}}$ do (this is not too difficult to see because of weight considerations), but $t_1$ does not act as $s_{\{1\}}$, in fact, we conjecture that $t_1$ acts as the element $s_{\{2\}}s_{\{1\}}s_{\{1,2\}}$, and we have many similar conjectures. Some of these we can verify; for example, a tedious, albeit elementary, computation shows that the element $t_4t_2t_4t_2t_4$ acts as the element $s_{\{2\}}$.

Based on computational evidence, we conjecture the following:

\begin{conjecture}
  For type $D_m$ and the crystal $B(n\omega_1)$, the action of the cactus group is generated by the action of cactus generators corresponding to one and two element subsets of $I$.
\end{conjecture}

This conjecture does not hold for other type $D_m$ crystals. For representations associated to spin nodes $\omega_{m-1}$ and $\omega_m$, some unpublished results of Frieden and Thomas can be used to show that
\begin{align*}
t_1 &\sim s_{\{1\}} \\
t_2 &\sim s_{\{1\}}s_{\{1, 2\}} s_{\{1\}} \\    
t_k &\sim s_{\{1, 2, \ldots, k-1 \}} s_{\{1,2,\ldots , k\}} s_{\{1, 2, \ldots, k-1 \}} s_{\{1, 2, \ldots, k-2 \}}\text{ for }3\leq k \leq m-2,
\end{align*}
where $\sim$ means that the elements of the toggle and cactus groups act in the same way on the crystals $B(n\omega_{m-1})$ and $B(n\omega_m)$. Notice that $t_1$ acts as the element $s_{\{1\}}$ in this case.

\section{Heaps and preprojective algebra modules}
\label{sec:lambdamod}

\subsection{Preprojective algebra and quiver Grassmannians}
\label{subsec:preprojquivergrs}
Let $Q=(I,E)$ be an orientation of the Dynkin diagram of $ \fg$. To each arrow $a:i\to j$ in $E$ we associate an arrow $a^\ast:j\to i$ going in the opposite direction, and define $E^\ast = \{a^\ast : a \in E\}$. 
Let $\overline{E} = E\sqcup E^\ast$ and extend ${}^\ast$ to $\overline{E}$ by declaring $(a^\ast)^\ast = a$. Let $\epsilon:\overline{E}\to\{\pm 1\}$ be defined by $\epsilon(a)=1$ and $\epsilon(a^\ast)=-1$.
The data of $(I,\overline{E},\epsilon)$ is called the double of $Q$ and denoted $\overline Q$.  For any arrow $a : i \to j$ in $ \overline{E} $, we will write $\text{tail}(a) = i$, $\text{head}(a) = j$.
\begin{definition}\label{def:preproj}
	The preprojective algebra $\Pi$ is the quotient of the path algebra of $\overline{Q}$ by the relation 
  \smash{$\sum_{a\in \overline{E}} \epsilon(a) a a^\ast = 0$}.
\end{definition}

A $ \Pi$-module $M $ can be described using the following data.  First, we have a vector space decomposition $ M = \bigoplus_{i\in I} M_i $ (this comes from the ``zero-length'' paths).  Second, we have linear maps $ a : M_i \rightarrow M_j $ for each arrow $ a : i \to j $ in $ \overline{E} $.  
These linear maps have to satisfy the preprojective algebra relation, which imposes that at each vertex $ i \in I $, we have
$$
  \sum_{\substack{a\in \overline{E}\\ \text{head}(a) = i}} \epsilon(a) a a^\ast = 0 \,. 
$$
 Given a $ \Pi$-module $M$, we record the dimensions of all $ M_i$ as a vector $\dimvec M \in \bN^I$.

The simple $\Pi$-modules are the 1-dimensional modules $ S(i) $ consisting of a 1-dimensional vector space at vertex $ i $ and 0's elsewhere (and of course all the arrows act by 0).  The \textbf{socle} $ \soc(M)$ of a $ \Pi$-module $ M$ is the maximal semisimple submodule of $ M $.

Every simple module $ S(i) $ has an injective hull denoted $T(i)$.  
$T(i)$ is the unique (up to isomorphism) $ \Pi$-module with dimension vector $ \bv = (v_i)$, such that $ \omega_i - w_0 \omega_i = \sum v_j \alpha_j $, and socle $ S(i) $.

More generally, for each $ w \in W $, by \cite[Theorem 3.1]{mvbasis}, there is a unique (up to isomorphism) $ \Pi$-module $T(i,w) $ with dimension vector $ \bv $, where $ \omega_i - w \omega_i = \sum v_j \alpha_j $, and socle $ S(i)$ (unless $ w \omega_i = \omega_i $, in which case $ T(i,w) = 0 $).

Let $ \lambda $ be a dominant weight and write $ \lambda = \sum n_i \omega_i $.  Associated to such $ \lambda $, we consider the direct sums $ T(\lambda) := \bigoplus_i T(i)^{\oplus n_i} $ and $ T(\lambda, w): = \bigoplus_i T(i,w)^{\oplus n_i} $.

Note that $ T(\lambda, w) $ is the unique module of dimension vector  $ \bv $, where $ \lambda - w \lambda = \sum v_j \alpha_j $, and socle contained in $ \oplus_i S(i)^{\oplus n_i}$.  Note that $ T(\lambda, w) $ is a submodule of $ T(\lambda) $, since $ T(\lambda) $ is the injective hull of $ \oplus_i S(i)^{\oplus n_i}$.

We will study the quiver Grassmannians
$$
  \Gr(T(\lambda)) := \left \{ M \subset T(\lambda) : M \text{ is a $ \Pi$-submodule} \right\}
$$
and $ \Gr(T(\lambda,w))$, defined similarly.  
These spaces have connected components $ \Gr(\bv, T(\lambda))$ and $\Gr(\bv, T(\lambda,w)) $ resp.\ where $ \bv \in \bN^I $ records the dimension vector of $ M $.

The following result follows by transporting the Saito \cite{Sai02} crystal structure on the irreducible components of cores of Nakajima quiver varieties (and its extension to Demazure crystals \cite[Prop 6.1]{Sav06}) with the Savage--Tingley isomorphism \cite[Prop 4.12]{savage2011quiver}. 

\begin{theorem} 
  \label{thm:DemCrystalGeometric}
	\begin{enumerate} \item There is a crystal structure on $ \Irr \Gr(T(\lambda)) $ making it isomorphic to $B(\lambda)$. 
		\item The inclusion $ T(\lambda, w) \subset T(\lambda)$, gives rise to an inclusion $ \Irr \Gr(T(\lambda, w)) \subset \Irr \Gr(T(\lambda)) $.  
		\item The crystal structure on $ \Irr \Gr(T(\lambda)) $ restricts to a crystal structure on $ \Irr \Gr(T(\lambda,w)) $ identifying it with the Demazure crystal $ B_w(\lambda) \subset B(\lambda)$.
	\end{enumerate}
\end{theorem}

\subsection{Orientations and colouring}
\label{subsec:orientations}

We now enter a short combinatorial interlude to discuss a certain 4-colouring associated to a heap; this discussion is needed to choose signs appropriately when constructing the $\Pi$-module associated to a given heap in the next subsection.

First, fix a $2$-colouring $c' : I \to \{-, +\}$, which is possible since all Dynkin diagrams are trees. Then, choose the orientation of the Dynkin diagram such that all arrows go from $ (-) $-vertices to $ (+)$-vertices. 

We will also need the following lemma. Recall that $H(s_{i_1}, \dots, s_{i_\ell})$ can be pictured as the dominant minuscule heap obtained by dropping beads $b_{j}$ onto the runners $i_j$ beginning with $b_\ell$ and proceeding in order until $b_1$ is dropped. Note that at the $j$th step when the bead $b_j$ is dropped, it covers previously-dropped beads $y_1$ (and possibly $y_2$) on runners which neighbour the runner $i_j$. (It covers at most two beads by \Cref{dommin-condition}.)

\begin{definition}
In the setting above, we say $b_j$ is ``in good order" if either $y_2$ doesn't exist, or $y_1$ and $y_2$ both exist but one of them is maximal in the sub order ideal $\{b_{j+1}, \dots, b_\ell\}$.
\end{definition}
\begin{lemma}
\label{lem:goodorder}
Let $H(w)$ be a minuscule heap. Then we can choose a reduced word $ \bw = s_{i_1}, \dots, s_{i_\ell} $, so that all beads are dropped in good order in $ H(s_{i_1}, \dots, s_{i_\ell})$.
\end{lemma}

\begin{proof}
Recall that we have a rank function on $H(w) $ which we refer to as level (see \cref{prop:rank}). So we will begin by choosing to drop beads in a way such that we drop all the beads on any given level $k$ before moving onto level $k + 1$. It remains now to describe the order in which to drop the beads of a given level $k$.

Now let $B_k$ be the graph whose nodes are all of the beads of level $k$ in the full poset $H(w)$, with edges between two beads $b$ and $b'$ if there exists a bead $b''$ of level $k-1$ in $H(w)$ such that $b$ and $b'$ both cover $b''$. By \Cref{thm:lminuscule}, each bead in $B_k$ can cover at most two beads in $B_{k-1}$, and therefore each vertex in $B_k$ has degree at most two. It is therefore a disjoint union of linear graphs. Note that, by the definition of good order, a bead $b_j$ of level $k$ fails to be dropped in good order if and only if it has two neighbours in $B_k$, both of which have already been dropped before $b_j$ itself. Since $B_k$ is a union of linear graphs, we can drop beads in an order such that this never happens: at each step, drop a bead such that it has exactly one un-dropped neighbour in $B_k$. Clearly this achieves the desired property, by induction.
\end{proof}

In the subsequent discussion, we will consider $H(w)$ as a graph where the nodes are beads and the edges are given by covering relations in the poset.

\begin{proposition}
\label{prop:colouringconditions}
Given a $2$-colouring $c'$ of $I$ as above, for any dominant minuscule heap $H(w)$, there exists a $4$-colouring $c$ of the edges (i.e. covering relations) in $H(w)$ with colours $\{R, B, G, Y\}$ (red, blue, green, and yellow) which satisfies the properties
\begin{enumerate}
\item $c$ is indeed a $4$-colouring, i.e. for any bead $x \in H(w)$, the edges of $H(w)$ which touch $x$ are all of different colours. 
\item If $e$ is an edge which serves as a covering relation between two beads $x \in H(w)_i$ and $y \in H(w)_j$ for $i, j \in I$ adjacent runners, with $x > y$, then $c(e) \in \{R, B\}$ (resp.\ $c(e) \in \{G, Y\}$) if $c'(i) = -$ (resp.\ $c'(i) = +$).
\item Any four edges which form a \emph{diamond} (i.e. a subgraph of $H(w)$ of the form pictured in \Cref{fig:smallerheap}) are of different colours.
\end{enumerate}
\end{proposition}

\begin{figure}[h!]
  \[
    \begin{tikzcd}[column sep = small, row sep = small]
    & & 0 \ar[d] & & & \\
    1 & 2 \ar[r] \ar[l] & 3 & 4 \ar[r] \ar[l] & 5 & 6 \ar[l]
    \end{tikzcd}
    \hspace{1cm} 
    \adjustbox{scale=0.75}{\begin{tikzcd}[column sep = tiny, row sep = tiny]
    & & & & & 6 \ar[dl, blue]\\
    & & & & 5 \ar[dl, Green] & \\
    & & & 4 \ar[dl, blue] & & \\
    & & 3 \ar[dr, Dandelion] \ar[dl, Green] & & & \\
    & 2 \ar[dl, blue]\ar[dr, red] & & 0 \ar[dl, blue] & & \\
    1 \ar[dr, Dandelion] & & 3 \ar[dl, Green] \ar[dr, Dandelion]& & & \\
    & 2\ar[dr, red] & & 4 \ar[dl, blue]\ar[dr, red] & & \\
    & & 3 \ar[dl, Green]\ar[dr, Dandelion] & & 5 \ar[dl, Green]\ar[dr, Dandelion] & \\
    & 0\ar[dr, red] & & 4 \ar[dl, blue]\ar[dr, red] & & 6 \ar[dl, blue]\\
    & & 3 \ar[dl, Green] \ar[dr, Dandelion]& & 5\ar[dl, Green] & \\
    & 2 \ar[dl, blue] \ar[dr, red]& & 4 \ar[dl, blue] & & \\
    1 \ar[dr, Dandelion] & & 3 \ar[dl, Green] \ar[dr, Dandelion]& & & \\
    & 2 \ar[dr, red] & & 0 \ar[dl, blue] & & \\
    & & 3 \ar[dr, Dandelion] & & & \\
    & & & 4 \ar[dr, red] & & \\
    & & & & 5 \ar[dr, Dandelion] & \\
    & & & & & 6 \\
    \end{tikzcd}}
  \]
    \caption{A $4$-colouring of the edges of a heap satisfying the requirements in \Cref{prop:colouringconditions}, with a label on each bead given by the label of the vertex over which it lies.\label{fig:4colouring}}
\end{figure}

\begin{proof}
Choose an ordering of beads such that $H(w)$ is formed by dropping these beads in order, and such that at each time a bead is dropped, it is dropped in good order. We will construct the colouring $c$ inductively; suppose for induction that $b_\ell, \dots, b_{j+1}$ have already been dropped, and that the edges between these beads already satisfy the three properties in the proposition.

When bead $b_j$ is dropped, it covers either one bead $y_1$, forming a new edge $e_1$, or two beads $y_1$ and $y_2$, forming new edges $e_1$ and $e_2$ respectively. Without loss of generality, suppose $c'(i_j) = +$.  In the first case where only $y_1$ exists, we can give $e_1$ an arbitrary choice of colour so long as $c(e_1) \in \{G, Y\}$. In the second case, without loss of generality, by the ``good order" condition, $y_2$ is maximal among the beads $\{b_\ell, \dots, b_{j+1}\}$. If there exists a bead $b_r \in \{b_\ell, \dots, b_{j+1}\}$ which covers $y_1$, then the edge $e$ corresponding to the covering relation $b_r > y_1$ satisfies $c(e) \in \{G, Y\}$, since $c'(i_j) = +$ and $y_1$ lies on a  runner neighbouring $i_j$, with $e$ emanating downward from a bead neighbouring the runner of $y_1$ itself. In this case, colour $e_1$ and $e_2$ opposite colours in $\{G, Y\}$ such that $c(e_2) = c(e)$. If instead no bead among $\{b_\ell, \dots, b_{j+1}\}$ covers $y_1$, then give $e_1$ and $e_2$ any choice of opposite colours among $\{G, Y\}$.

After $b_j$ is dropped and the choices of colouring on $e_1$ (and maybe $e_2$) as described are made, Property 2 is clearly satisfied by construction. Property 3 is clear, since if a new diamond is formed by adding $b_j$, then the edges of this diamond are now coloured by $\{c(e_1), c(e_2)\} = \{G, Y\}$ along with two edges which emanate downward from $y_1$ and $y_2$ and meet at a single bead on runner $i_j$. These edges are coloured by $R$ and $B$ in some order, by the Properties 1 and 2 which hold by our inductive hypothesis. Finally, Property 1 is clear since no edges coloured by $G$ or $Y$ other than the ones considered in the previous paragraph can touch $y_1$ or $y_2$ by the ``good order" property, along with the fact that by \Cref{thm:lminuscule} most two beads can cover a given bead in any dominant minuscule heap.
\end{proof}

\begin{example}
  Pictured in \Cref{fig:4colouring} above is an example of a $4$-colouring of edges satisfying the conditions of the proposition for a heap in type $E_7$. Here, the $2$-colouring of the Dynkin diagram is the unique $2$-colouring with $c'(1) = +$, and the corresponding orientation is pictured on the left. A (non-unique) $4$-colouring $c$ of edges compatible with $c'$ is pictured on the right.
\end{example} 

\subsection{Preprojective algebra modules from heaps}
\label{subsec:preprojalgmodsfromheaps}

In the previous section, we fixed an orientation of the Dynkin diagram.  Now and for the rest of the paper, we will assume that the preprojective algebra is defined with respect to this orientation.

For any minuscule $w \in W$, our goal is to define a module over the preprojective algebra  using the heap $H(w)$.  Recall the map $ \pi : H(w) \rightarrow I $ and its fibres $ H(w)_i := \pi^{-1}(i) $.  
We build out of this data an $I$-graded vector space $ \bC H(w) $ with basis $ H(w) $. 
Choosing a $4$-colouring $c$ of the edges in $H(w)$ satisfying  properties detailed in the previous subsection and specifically \cref{prop:colouringconditions} we now promote this vector space to a graded module for the path algebra of $\overline Q$. 

Fix an arrow $a = i \to j$ in $\overline{E}$ and $ x \in H(w)_i $.  If there exists $ y \in H(w)_j $ such that $ y < x $ is a covering relation $e$ in $ H(w) $, then we  define $ a(x) = \sigma(c(e))y $ where $\sigma : \{R, B, G, Y\} \to \{-1, 1\}$ sends $R$ to $-1$ and $\{B, G, Y\}$ to $1$ ; if no such $ y $ exists then we set $a(x) = 0 $.  (Such a $ y $ is necessarily unique, since $ H(w)_j $ is totally ordered.)

This module can be drawn by picturing $H(w)$ via the beads-and-runners interpretation discussed earlier, where each bead corresponds to a basis vector and an arrow $ a $ maps a bead on runner $i$ to a bead that it touches on runner $ j$.
\begin{example}
Consider the heap $H(s_2, s_3, s_1, s_2)$ in type $A_3$ as in \Cref{example:heap}. 
After choosing the colouring $c$ pictured below on the left, the module structure on $\bC H(w)$ for this choice of colouring can be written as below on the right:
$$
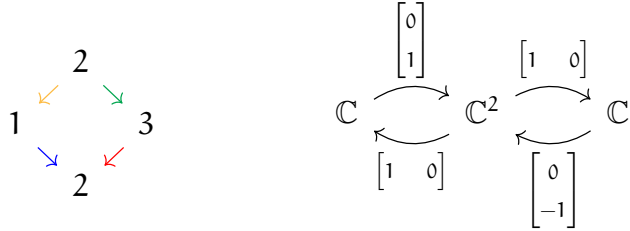
\begin{tikzcd}[column sep = tiny, row sep = tiny]
                & 2\ar[dr, Green]\ar[dl, Dandelion] &               \\
1\ar[dr, blue]  &                                   & 3\ar[dl, red] \\
                & 2                                 & 
\end{tikzcd}
\hspace{2cm}
\begin{tikzcd}
\mathbb{C} \ar[r, "\begin{bmatrix}0\\ 1\end{bmatrix}", bend left] & \mathbb{C}^2 \ar[r, "\begin{bmatrix}1 \quad 0\end{bmatrix}", bend left] \ar[l, "\begin{bmatrix}1 \quad 0\end{bmatrix}", bend left] & \mathbb{C}\ar[l, "\begin{bmatrix}0\\ -1\end{bmatrix}", bend left] 
\end{tikzcd}
$$
where a basis for each $\mathbb{C}H(w)_i$ is chosen so that basis elements correspond to beads on runner $i$ in $H(w)$.
\end{example}

\begin{proposition}\label{thm:dommin-preproj}
If $ w $ is minuscule, then the preprojective relation is satisfied on $ \bC H(w) $ and so  $\bC H(w)$ has the structure of a $\Pi$-module.
\end{proposition}

\begin{proof} Recall that \Cref{dommin-condition} gives us explicit conditions on $H(w)$ which are equivalent to $w$ being minuscule.  We must verify that $\bC H(w)$ satisfies the condition for being a module over the preprojective algebra. By \Cref{def:preproj}, we must verify that for any $i \in I$,
\begin{align*}
\sum_{\substack{a \in \overline{E}\\ \mathrm{head}(a) = i}} \epsilon(a) aa^* = 0.
\end{align*}

If $x$ is minimal in $H(w)_i$, then this is clear, since by the definition of the module structure on $\bC H(w)$, we have $aa^*(x) = 0$ for all $a$ such that $ \mathrm{head}(a) = i$. Otherwise, choose a maximal $y \in H(w)_i$ satisfying $y < x$. 
By \Cref{thm:lminuscule}, there are exactly two heap elements $z_1$ and $z_2$ on vertices neighboring $i$ which are in between $y$ and $x$ in the partial order on $H(w)$. These elements are such that either $\pi(z_1) \neq \pi(z_2)$ or $\pi(z_1) = \pi(z_2)$.

In the first case, by the characterization of intervals in heaps between elements over the same vertex provided in \cite{St01}, we know that $x$ covers both $z_1$ and $z_2$, which both cover $y$; this shows that $\{x, z_1, z_2, y\}$ form a diamond. So let $a_1 : \pi(z_1) \to i$ and $a_2 : \pi(z_2) \to i$. By our choice of orientation of $Q$ constructed from the $2$-colouring $c'$, we know that either both $a_1, a_2 \in E$ (in the case where $c'(i) = +$), or both $a_1, a_2 \in E^*$ (if $c'(i) = -$). We will handle the first case, with the second case being handled identically. Since $x, z_1, z_2, y$ form a diamond, if we let $c$ be the colouring which used to construct the $\Pi$-module, then one of the edges between these four heap elements has colour $R$. We will handle the case where the edge from $z_1$ to $x$ has this colour, with the other cases being handled similarly. In this setup, we get
\begin{align*}
\smash{\sum_{\substack{a \in \overline{E}\\ \mathrm{head}(a) = i}} \epsilon(a) aa^*}(x) 
      &= (\pi(z_1) \to i)(i \to \pi(z_2))(x) + (\pi(z_2) \to i)(i \to \pi(z_1))(x)\\
      &= (\pi(z_1) \to i)(z_1) + (\pi(z_2) \to i)(z_2)\\
      &= -x + x\\
      &= 0.
\end{align*}

If instead $\pi(z_1) = \pi(z_2)$, then without loss of generality we can assume that $z_1 > z_2$. This means if $a^*$ is the arrow $i \to \pi(z_1)$, we have $a^*(x) = \pm z_1$. But since $y < z_2 < z_1$ by assumption, $z_1$ does not cover $y$, and so $aa^*(x) = 0$. Further, if $a^*$ is an arrow $i \to j$ where $j \neq \pi(z_1)$, we clearly have $a^*(x) = 0$ to begin with.
\end{proof}

\begin{figure}[ht]
\centering
\begin{tikzcd}[column sep = tiny, row sep = tiny]
              & x \ar[ddl]\ar[ddr] &            \\
                                                \\
z_1 \ar[ddr]  &                    & z_2\ar[ddl]\\
                                                \\
              & y
\end{tikzcd}
\hspace{2cm}
\begin{tikzcd}[column sep = tiny, row sep = tiny]
x \ar[dr] &             & \\
          & z_1 \ar[dr] & \\
          &             & {\mathmakebox[0pt][l]{\smash{\vdots}}\phantom{z_3}} \ar[dl] \\  
          & z_2 \ar[dl] & \\
y         &             & 
\end{tikzcd}
\caption{The first and second cases in the proof of \Cref{thm:dommin-preproj}.}
\end{figure}

\begin{remark}
In Type $A_{m-1}$, an alternative orientation can be chosen for which no sign modifications are needed, making the content of \cref{subsec:orientations} unnecessary. Indeed, in this case we use the orientation $ 1 \to 2 \to \dots \to m-1 $. Then we define the action of the path algebra of $\overline{Q}$ on $\mathbb{C}H(w)$ by setting $a(x) = y$ whenever $a = i \to j \in \overline{E}$ and $x \in H(w)_i, y \in H(w)_j$ and $y < x$ is a covering relation, omitting the $\sigma$ in our original definition of $\mathbb{C}H(w)$. 

One can then check that \Cref{thm:dommin-preproj} holds with this definition, essentially because every sub $A_3$ Dynkin diagram is linearly oriented with orientation. However, in other types, there is a trivalent vertex and so there is no orientation which linearly orients every sub $A_3$.  That is why we needed a careful choice of signs as in our original definition of the $\Pi$-module $\mathbb{C}H(w)$.
\end{remark}

When $ w $ is dominant minuscule, then this module gives us a concrete realization of some of the modules defined abstractly in the last section.

\begin{lemma}\label{thm:projcover}
Let $w \in W $ be dominant minuscule and $ \lambda $ be a witness for $ w $.  Then $ \bC H(w) \cong T(\lambda, w) $.	In particular, if $ \lambda $ is minuscule and $ J = \left\{ j : s_j \lambda = \lambda \right\} $, then $ \bC H(w^J_0) = T(\lambda)$.
\end{lemma}

\begin{proof}
Let $ \bv = \dimvec \bC H(w)$.  Choose a reduced word $ w = (s_{i_1}, \dots, s_{i_\ell} )$.  Then, by the construction of $ H(w)$, we see that $ \sum v_i \alpha_i = \sum_k \alpha_{i_k} $.  Since $ \lambda $ is a witness for $ w $, we see that $ \sum_k \alpha_{i_k} = \lambda - w \lambda $.  So $ \bC H(w) $ has the right dimension vector.

The socle of $ \bC H(w) $ is spanned by the minimal elements of the heap $ H(w) $.  By the construction of the heap, these minimal elements are in bijection with  $\{ i\in I : ws_i < w \} $ (as these are the possible rightmost generators in a reduced word for $ w $).  By \Cref{thm:witness} and the characterization of $ T(\lambda, w)$, the result follows.
\end{proof}
\begin{example}
In type $D_4$, let $w = s_2s_1s_3s_4s_2$. 
In this case, the module $\bC H(w)$ can be depicted as
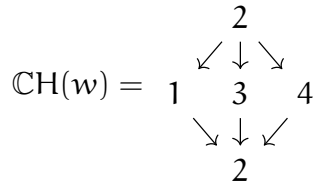
\[
    \bC H(w)=\begin{tikzcd}[column sep = tiny, row sep = tiny]
    & 2 \ar[ddl]\ar[dd]\ar[ddr] &\\
    \\
    1 \ar[ddr] & 3 \ar[dd] & 4\ar[ddl]\\
    \\
    & 2
    \end{tikzcd}
\]
where each label denotes a basis element of the vector space supported at the corresponding numbered vertex, and the arrows indicate the actions on these basis elements by the path algebra of the $D_4$ doubled quiver.

Note that in this case, the criterion in \Cref{thm:lminuscule} tells us that $w$ is not minuscule.  There is no way to make $ \bC H(w)$ into a $ \Pi$-module where the arrows of $ \overline E $ act by the covering relations in the heap, up to sign.  (It is possible to make $ \bC H(w) $ into a $ \Pi$-module if we use other scalars.)
\end{example}

\begin{example}
  In type $D_5$, let $w = s_5s_3s_2s_4s_1s_3s_2s_5s_3s_4$. For one choice of colouring, the module $\bC H(w)$ can be depicted as follows.
  \[
    \bC H(w)=
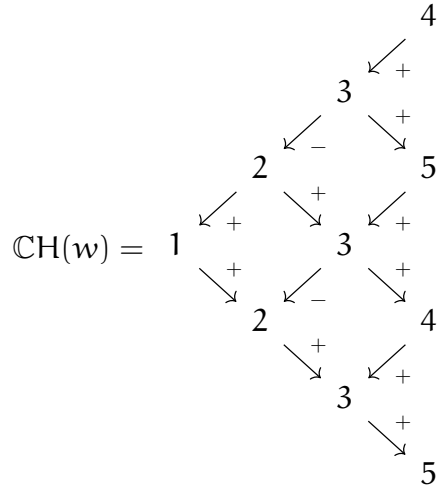
\begin{tikzcd}[column sep = small, row sep = small]
      & & & 4 \ar[dl, "+"]\\
      & & 3 \ar[dl, "-"] \ar[dr, "+"]& \\
      & 2 \ar[dl, "+"] \ar[dr, "+"]& & 5 \ar[dl, "+"]\\
      1 \ar[dr, "+"] & & 3 \ar[dl, "-"] \ar[dr, "+"]& \\
      & 2 \ar[dr, "+"] & & 4 \ar[dl, "+"]\\
      & & 3 \ar[dr, "+"] & \\
      & & & 5\\
    \end{tikzcd}
  \]
  In this case we can easily verify using \Cref{dommin-condition} that $w$ is dominant minuscule, and also that this module is a $\Pi$-module, as is guaranteed by \Cref{thm:dommin-preproj}. Finally, note that this is also an example where \Cref{thm:projcover} holds, as we can verify that $\bC H(w)$ is isomorphic to the injective hull of $S(5)$ while $w = w_0^J$ for $J = \{1, 2, 3, 4 \}$.
\end{example}

An important property of $ \bC H(w) $ is that it has finitely many submodules.  In fact, we have the following result.
\begin{proposition} \label{th:GrHw}
    Let $ \phi \subset H(w)$ be an order ideal. Then $ \bC \phi $ is a submodule of $ \bC H(w) $ and $\phi\mapsto \bC\phi$ defines a bijection $ J(H(w)) \rightarrow \Irr \Gr (\bC H(w))$.
\end{proposition}

\subsection{Nilpotent endomorphisms}\label{sec:nilp}
Fix a dominant minuscule $ w $ with witness $ \lambda $ and $ n \in \bN $. Combining \Cref{thm:projcover}, \Cref{thm:DemCrystalGeometric}, and \Cref{thm:gravsort}, we see that $ \Irr Gr(\bC H(w)^{\oplus n}) $ and $ \Rpp(w,n) $ are both models for the Demazure crystal $ B_w(n\lambda)$.  We will now build the bijection between these models.

We define, for each $i \in I$, a nilpotent linear endomorphism $\tilde{A}_i$ of $\bC H(w)_i$. We note in advance that together, the collection $\{\tilde{A}_i\}_{i\in I}$ will not define an endomorphism of $\bC H(w)$ as a $\Pi$-module. By \Cref{rmk:runners}, each $H(w)_i$ is totally ordered so we can enumerate $H(w)_i =\{x_i^1 < \dots < x_i^q\}$.
For each $i\in I$, let $\tilde{A}_i : \bC H(w)_i \to \bC H(w)_i$ to be the downward shift operator, defined on this basis by 
$$\tilde{A}_i(x_i^s) = x_i^{s-1} \text{ for } 1 < s \leq q, \text{ and } \  \tilde{A}_i(x_i^1) = 0\,. $$

Since $\tilde{A}_i$ is a linear endomorphism of $\bC H(w)_i$, it extends in the natural way to a linear endomorphism $\tilde{A}_i$ of $\bC H(w)_i^{\oplus n}$. We now show that in Types $A$ and $D$, $\tilde{A}_i$ always coincides with the action of some algebra element $A_i \in \Pi$. This will imply that if $M$ is a submodule of $\bC H(w)^{\oplus n}$, then $M_i$ is invariant under the action of $ \tilde{A}_i $ (because it is invariant under $ \tilde A_i \in \Pi $).

\begin{lemma}\label{lem:agreeswithpi}
  In Types $A$ and $D$, there exists $A_i \in \Pi$ such that $A_i = \tilde{A}_i$ as linear endomorphisms of $\mathbb{C}H(w)_i$.
\end{lemma}

\begin{proof}
  We check this case-by-case for the three classes of posets in Figure \ref{fig:typedposets}. In Class 1, the composition of an arrow from vertex $i$ to $i \pm 1$ and back again will always work (one of these will give $\tilde{A}_i$ and the other $-\tilde{A}_i$). Similarly, this same choice works for all vertices aside from $n-1$ and $n$ in Class 2 (for Type $D_n$ with $n \geq 4$). For vertex $n - 1$ in Class $2$, one can check that the two paths
  \begin{align*}
    n-1 \to n - 2 \to n - 3 \to n - 2 \to n - 1,\\
    n-1 \to n - 2 \to n \to n - 2 \to n - 1
  \end{align*}
  correspond to the operators $\tilde{A}_i$ and $-\tilde{A}_i$ in some order. The same idea works for vertex $n$. Finally, in Class 4, the two paths
  \begin{align*}
    i \to i + 1 \to \dots \to n - 2 \to n - 1 \to n - 2 \to \dots \to i + 1 \to i,\\
    i \to i + 1 \to \dots \to n - 2 \to n \to n - 2 \to \dots \to i + 1 \to i
  \end{align*}
  also yield $\tilde{A}_i$ and $-\tilde{A}_i$ in some order, and so in each of these cases we choose the one which yields $\tilde{A}_i$.
\end{proof}

\begin{figure}
  \caption{An example for $n = 5$ of Class 1 (left) for Type $A_n$ and  Classes 2 (middle) and 4 (right) for Type $D_n$ of the posets described in Section 7 of \cite{P992}\label{fig:typedposets}; these are the only classes which can arise outside of Type $E$.}
  \[\begin{tikzcd}[column sep = small, row sep = small]
    & & 3 \ar[dl] \ar[dr]& & \\
    & 2 \ar[dl] \ar[dr]& & 4 \ar[dl] \ar[dr] & \\
    1 \ar[dr] & & 3 \ar[dl] \ar[dr]& & 5 \ar[dl]\\
    & 2 \ar[dr] & & 4 \ar[dl]& \\
    & & 3 & & \\
  \end{tikzcd}
  \hspace{1.5cm} \begin{tikzcd}[column sep = small, row sep = small]
    & & & 4 \ar[dl]\\
    & & 3 \ar[dl] \ar[dr]& \\
    & 2 \ar[dl] \ar[dr]& & 5 \ar[dl]\\
    1 \ar[dr] & & 3 \ar[dl] \ar[dr]& \\
    & 2 \ar[dr] & & 4 \ar[dl]\\
    & & 3 \ar[dr] & \\
    & & & 5\\
  \end{tikzcd} \hspace{1.5cm} \begin{tikzcd}[column sep = small, row sep = small]
    & 1 \ar[dr] & & & \\
    & & 2 \ar[dr] & & \\
    & & & 3 \ar[dl] \ar[dr] & \\
    & & 4 \ar[dr] & & 5 \ar[dl]\\
    & & & 3 \ar[dl] & \\
    & & 2 \ar[dl] & & \\
    & 1 & & & 
  \end{tikzcd}\]
\end{figure}

\begin{remark}
  In Type E, it may not always be true that $\tilde{A}_i$, as defined above, agrees with the action of an element of $\Pi$ as shown in Lemma \ref{lem:agreeswithpi}. We believe, however, that Lemma \ref{lem:agreeswithpi} is true ``up to sign," i.e. that there exists $A_i \in \Pi$ such that for any $k$ with $1 \leq k \leq r_i-1$, $A_i(x_i^s) = \pm x_i^{s-1}$. We believe that this can be checked manually for the posets in Classes 3 through 15 of \cite{P992}. 
  
  However, the only properties of $A_i$ which are relevant for the arguments in this paper are that $A_i \in \Pi$ and $\ker A_i^s = span (x_i^1, \dots, x_i^s)$ for $1 < s \leq q$. These properties are just as well satisfied by a generic linear combination of paths from the vertex $i$ to itself. To construct such a linear combination, note that since the posets we consider are finite, for any $i$, there exists some number $R_i$ such that the image in $\Pi$ of any path in $\overline{Q}$ of length greater than $R_i$ sends $x_i^s$ to zero for all $1 \leq s \leq q$. The image in $\Pi$ of a generic linear combination of paths of length at most $R_i$ in $\overline{Q}$ from $i$ to $i$ will represent an element $A_i \in \Pi$ satisfying $\ker A_i^s = span (x_i^1, \dots, x_i^s)$ for $1 < s \leq q$. This follows from the fact that (for any of the heaps we consider), for any $1 < s \leq q$, there always exists at least one path from $x_i^s$ to $x_i^{s-1}$. So a generic linear combination will take each $x_i^s$ to some nonzero multiple of $x_i^{s-1}$ plus some linear combination of $x_i^j$ for $j < s$.

So in Type $E$, rather than using the more direct combinatorial definition of the endomorphism $A_i$ as a ``downward-shift operator up to sign," we instead make such a generic choice of $A_i \in \Pi$ for each $i$, and use this as our definition. (One can just as easily take this as the definition in all types and disregard Lemma \ref{lem:agreeswithpi} altogether, but we include it to show that a more direct combinatorial description is often possible.)
\end{remark}

We define a function $ \Phi_M : H(w) \rightarrow \{0, \dots, n\} $ which records the Jordan forms of each $ A_i\big|_{M_i} $ as follows.
$$
  \Phi_M(x_i^s) := \dim (\ker A_i^s \cap M_i) - \dim (\ker A_i^{s-1} \cap M_i)
$$
With this definition, the values $\Phi_M(x_i^1), \dots, \Phi_M(x_i^q) $ represent the dual partition to the partition made up of the sizes of the Jordan blocks of $  A_i\big|_{M_i} $.

Let $ Z \in \Irr \Gr(\bC H(w)^{\oplus n}) $. Since $ M \mapsto \Phi_M $ is a constructible function, it is constant on a dense constructible subset of $ Z $. We can therefore define $ \Phi_Z := \Phi_M $, for $ M $ a general point of this subset. We can now state precisely our main result, which we will prove in \cref{sec:bijection}.
\begin{theorem} \label{thm:mainprecise}
For each $ Z \in \Irr \Gr(\bC H(w)^{\oplus n})$, $ \Phi_Z $ is a reverse plane partition.  
The map $Z \mapsto \Phi_Z$ defines a crystal isomorphism $ \Irr \Gr(\bC H(w)^{\oplus n}) \cong \Rpp(w, n) $.
\end{theorem}

\subsection{Comparison with components of Springer fibres} \label{se:Maffei}

It is instructive to compare \Cref{thm:mainprecise} with \Cref{th:classic}.  To do this, we will first need to recall an isomorphism of \cite{maffei2005quiver}, comparing the Springer fibre with the core of a Nakajima quiver variety.  

Fix $ \fg = \mathfrak{sl}_m $ and $ \lambda = \omega_p $. Let $ N = np $ and let $ A $ be a nilpotent operator on $ \bC^N $ with $ n $ Jordan blocks each of size $ p$.  Let $ F(A) $ denote the $m$-step Springer fibre for $ A $.  

On other hand, let $ w = w_0^J \in S_m $ (where $ J = \{1, \dots, m-1\} \setminus \{p\}$). For simplicity, we will assume that $ p \le m/2 $. The heap  $H(w) $ has a rectangular shape (as described in \cref{se:tableaux}) and $ H(w)_p$ has size $ p $.

Let $ L \in \Pi $ be the sum of all the left going arrows and $ R \in \Pi $ be the sum of all the right going arrows.  Note that in $ \bC H(w)$, $L : \bC H(w)_i \rightarrow \bC H(w)_{i-1} $ is injective for $ i > p $ and surjective for $ i \le p $.  Given $ M \in \Gr( \bC H(w)^{\oplus n}) $, we can construct an $m$-step flag inside of $ \bC H(w)^{\oplus n}_p$ by setting $ V_i = L^{m-p-i}(M_{m-i}) \subset \bC H(w)^{\oplus n}_p $, for $ i = 1, \dots, m-1$. 
If we identify $ \bC H(w)^{\oplus n}_p = \bC^N $ then $ A $ is given by $ A_p $ as defined in \cref{sec:nilp}.

Recall that in \cref{subsec:ytab}, we defined a semistandard Young tableau $ \Psi_V $ of shape $ (n^p) $ for any point $ V \in F(A) $.

\begin{theorem}\label{thm:flagrpp}
	\begin{enumerate}
	\item The map $ M \mapsto ( V_i)_{i = 1}^{m-1} $ defines an isomorphism $ \Gr ( \bC H(w)^{\oplus n}) \rightarrow F(A) $.  
	\item For each $ M \in \Gr ( \bC H(w)^{\oplus n}) $, the RPP $ \Phi_M $ agrees under the bijection from \cref{se:tableaux} with the tableau $\Psi_V $, up to applying the Sch\"utzenberger involution (see \Cref{rem:SI}).
	\end{enumerate}
\end{theorem}
From this theorem, we see that the description of the components of $ F(A) $ given in \Cref{th:classic} matches the description of the components of $ \Gr(\bC H(w)^{\oplus n}) $ from \Cref{thm:mainprecise}.

\subsection{Socle filtrations and RPPs} \label{se:socle}
We now give another interpretation of the bijection $\Irr \Gr(\bC H(w)^{\oplus n}) \cong \Rpp(w, n)$ which was suggested to us by Steven Karp and Hugh Thomas (see section 8 of their paper \cite{KT}).

Let $ M $ be a $\Pi$-module.  We define the socle filtration $ 0 \subset \soc(M) \subset \soc^2(M) \subset \cdots \subset M $ as follows.

\begin{definition}
First, recall that $soc (M)$ is the maximal semisimple submodule of M; so $\soc(M)_i \subset M_i$ is the kernel of all outgoing arrows from vertex $ i $.

Define the socle filtration of $M$ by requiring 
$$\soc^k(M) / \soc^{k-1}(M) = \soc( M / \soc^{k-1}(M)).$$ 
Equivalently, $ \soc^k(M)_i \subset M_i $ is the kernel of all paths of length $\ge k$ starting at vertex $i$ (to be more precise, such paths span a 2-sided ideal of $ \Pi $ and we are looking the annihilator of this ideal).
\end{definition}
The following simple observation will be used below.
\begin{lemma} \label{lem:socle}
The socle filtration behaves well with respect to submodules: if $ M' \subset M $ is a submodule, then $ \soc^k(M') = \soc^k(M) \cap M' $.
\end{lemma}

We define the \new{socle dimension matrix} of $M $ to be the map $ SD_M : I \times \bN \rightarrow \bN $ given by $ SD_M(i,k) = \dim \soc^k(M)_i - \dim \soc^{k-1}(M)_i $.  (Note that $ SD_M $ records the dimensions vectors of all subquotients $ \soc^k(M)/ \soc^{k-1}(M)$ and since they are semisimple, it uniquely specifies them.)

Given a dominant minuscule element $ w $ as before, we define a map $ H(w) \rightarrow I \times \bN $ using $ \pi $ in the first factor and using the level in the second factor.  This map is injective because the beads on any given runner are linearly ordered (and therefore appear on different levels).  So we can regard $ H(w) $ as a subset of $ I \times \bN $.  We can now compare the socle dimension matrix to the Jordan form information $ \Phi_M $ discussed previously.

\begin{theorem} \label{thm:SDPhi}
	Let $ M \subset \bC H(w)^{\oplus n }$.  Then $ SD_M|_{H(w)} = \Phi_M$ and $ SD_M $ is 0 outside of $ H(w) $.
\end{theorem}

Combining this result with Theorem \ref{thm:mainprecise}, we immediately deduce the following.

\begin{corollary}
The bijection $\Irr \Gr(\bC H(w)^{\oplus n}) \cong \Rpp(w, n)$ maps $ Z $ to $ SD_M $, where $ M$ is a general point of $ Z $.
\end{corollary}

To prove \Cref{thm:SDPhi}, we begin with the following lemma.
\begin{lemma} \label{lem:SDPhi}
  Let $ w $ be dominant minuscule.
	\begin{enumerate}
		\item $ \soc^k(\bC H(w)) $ is spanned by those beads of level less than or equal to $ k $.
		\item
	$SD_{\bC H(w)}$ is 1 on the image of $H(w) $ and $0 $ elsewhere.
	\item Let $ i \in I $, let $ x_i^s $ be a bead on runner $ i $, and let $ k(s) $ be the level of $ x_i^s$.  Then $\soc^{k(s)}(\bC H(w))_i = \ker A_i^s $.
	\end{enumerate}
\end{lemma}
\begin{proof}
	\begin{enumerate}
	\item If a bead $x$ has level $k $, then we can find a sequence $ x_1 = x > \cdots > x_k $ where each $ x_j > x_{j+1} $ is a covering relation.  By the definition of the module structure on $ \bC H(w) $ this gives us a path of arrows of length $ k-1 $ which takes $ x $ to $ x_k $ (up to a sign) and thus $ x $ does not lie in $ \soc^{k-1}(\bC H(w)) $.  On the other hand, if we have a path of arrows of length $ k $, then this would act by a sequence of $k $ covering relations.  As the bead $ x $ has level $ k $, no such sequence can exist.  Thus we conclude that $ x $ is killed by all paths of length $ k $ and thus lies in $ \soc^k(\bC H(w)) $.
	
	\item This follows immediately from part 1.
	\item From the definition of $ A_i$, we see that $ \ker A_i^s $ is the span of $ \{x_i^1, \dots, x_i^s \} $.  So the statement follows from part 1. 
		\end{enumerate}
\end{proof}

\begin{proof}[Proof of \Cref{thm:SDPhi}]
By Lemma \ref{lem:socle}, since $M \subset \bC H(w)^{\oplus n}$ we have $ \soc^k(M)_i = \soc^k(\bC H(w)^{\oplus n})_i \cap M_i $ for all $i$.  By the second part of the above \Cref{lem:SDPhi}, $ SD_{\bC H(w)^{\oplus n}} $ and hence also $ SD_M $ vanishes outside the image of $ H(w) $.  By the third part of the lemma, $ \soc^{k(s)}(\bC H(w)^{\oplus n})_i = \ker A_i^s $ where $ k(s) $ is the level of $ x_i^s $. 

Let $x_i^s\in H(w)_i$ with image $(i,k(s)) \in I \times \bN$. 
\begin{align*}
  SD_M(i,k(s)) &= \dim \soc^{k(s)} M_i - \dim \soc^{k(s)-1} M_i \\ 
               &= \dim \soc^{k(s)} (\bC H(w)^{\oplus n})_i\cap M_i - \dim \soc^{k(s)-1} (\bC H(w)^{\oplus n})_i\cap M_i \\
               &= \dim \ker A_i^s \cap M_i - \dim \ker A_i^{s-1} \cap M_i = \Phi_M(x_i^s)
\end{align*} 
where the third equation follows by the first part of the above lemma. More precisely, since $k(s-1) \le k(s) - 1$ and $\dim \soc^{k(s) - 1} (\bC H(w)^{\oplus n})_i$ is spanned by those beads of level less than or equal to $k(s) - 1$, it is spanned by $\{x_i^1,\dots,x_i^{s-1}\}$. 
\end{proof}

\section{Tensor product varieties via the Savage--Tingley isomorphism}
\label{sec:tensor}

\subsection{Cores and quiver Grassmannians}
The goal of this section is to study Nakajima's tensor product varieties from the viewpoint of the Savage--Tingley isomorphism.  

We begin by recalling the definition of the Nakajima quiver varieties.  Let $ \lambda \in P_+ $ and let $ \mu \in P $.  Write $ \lambda = \sum n_i \omega_i $ and $ \lambda - \mu = \sum v_i \alpha_i $ and assume that all $ v_i \in \bN $.

Let $ V$ and $ N $ be $ I$-graded vector spaces with dimensions $ \bv$ and $ \bn $ respectively.  We define the Nakajima quiver variety $ M(\bv, \bn) $ as follows.  We begin with
$$ 
  \Hom(\bv, \bn) := \bigoplus_{i \to j\in E} \Hom(V_i, V_j) \oplus \bigoplus_{i \in I} \Hom(V_i, N_i) 
$$ 
The cotangent bundle $T^*\!\Hom(\bv, \bn) $ comes with an action of the group $ GL(V) := \prod_i GL(V_i)$ with moment map 
$ \Psi : T^*\!\Hom(\bv, \bn) \rightarrow \mathfrak{gl}(V) $.  
We define $ M(\bv, \bn) := \Psi^{-1}(0) \sslash_{\chi} GL(V) $ where $ \chi $ is the product of the determinant characters.

There is a projective morphism $ M(\bv, \bn) \rightarrow M_0(\bv, \bn) :=  \Psi^{-1}(0) \sslash_0 GL(V) $.  The preimage of 0 under this morphism is called the core of the Nakajima quiver variety and is denoted $ L(\bv, \bn)$.  
Let $ M(\bn) := \bigsqcup_{\bv} M(\bv, \bn) $ and $ L(\bn) := \bigsqcup_{\bv} L(\bv, \bn) $.

We refer to the following result as the Savage--Tingley \cite[Theorem 4.4]{savage2011quiver} isomorphism (though the result is actually due to Lusztig \cite{lusztig1998quiver} and Shipman \cite{shipman2010representation}).
\begin{theorem} \label{th:STiso}
With notation as above, we have an isomorphism $ L( \bn) \cong Gr(T(\lambda)) $.
\end{theorem}

\subsection{A \texorpdfstring{$\Cx$}{C*}-action}
\label{subsec:cxaction}
Choose a splitting $\bn= \bn^1 + \dots + \bn^r $, which corresponds to a composition $ \lambda = \lambda^1 + \cdots + \lambda^r $. 
This defines a splitting $ N = N^1 \oplus \cdots \oplus N^r$ with $ \dimvec N^j = \bn^j $.  We define an action of $ \Cx $ on $ N $ by 
$$ s \cdot (u_1, \cdots, u_r) = ( s^{r-1} u_1, s^{r-2} u_2, \dots, u_r)\,. $$

This gives rise to an action of  $\Cx $ on both $ M( \bn) $ and $ L(\bn) $. 
From \cite[Lemma 3.2]{Nak01}, we have the following isomorphisms.
\begin{align}
    M(\bn)^\Cx &\cong M(\bn^1) \times \cdots \times M( \bn^r) \\
    L(\bn)^\Cx &\cong L(\bn^1) \times \cdots \times L( \bn^r) \label{eq:fixedpoints} 
\end{align}
We also get a splitting 
 $ T(\lambda) = T(\lambda^1) \oplus \cdots \oplus T(\lambda^r) $ 
and thus an action of $ \Cx $ on 
$T(\lambda)$ where $ s $ acts on 
$ T(\lambda^k) $ 
by $ s^{r-k} $, for each $k$.  
This induces an action of $ \Cx $ on $ Gr(T(\lambda)) $ and we have the evident isomorphism 
\begin{equation} \label{eq:fixedpoints2}
   Gr(T(\lambda))^{\Cx} \cong Gr(T(\lambda^1))\times \cdots \times Gr(T(\lambda^r))\,.   
\end{equation}

Thus we have defined an action of $ \Cx $ on both sides of the Savage--Tingley isomorphism.  This isomorphism is equivariant for these actions (by \cite[Prop 5.1]{savage2011quiver}) and thus yields an isomorphism between the fixed point sets compatible with the isomorphisms from \Cref{th:STiso} (see \cite[Thm 5.4]{savage2011quiver})
$$
  L(\bn^1) \times \cdots \times L(\bn^r) \cong Gr(T(\lambda^1))\times \cdots \times Gr( T(\lambda^r))
$$

\subsection{Tensor product varieties}
Following Nakajima, we consider the tensor product variety
$$
Z( \bn) = \left\{ x \in M(\bn) : \lim_{s\rightarrow 0} s \cdot x \in L( \bn)^\Cx \right\} \,. 
$$

According to Proposition 3.14 from \cite{Nak01}, there is a bijection 
\begin{equation} \begin{aligned} \label{eq:bijIrr}
   \Irr L(\bn^1) \times \cdots \times \Irr L(\bn^r) 
  &\to \Irr Z(\bn) \\
 (X_1, \dots, X_r) &\mapsto \overline{ \left \{ x \in M(\bn) : \lim_{s \to 0} s \cdot x \in X_1 \times \cdots \times X_r \right \}} \,.
\end{aligned}
\end{equation} 

Because $ L( \bn) $ is projective, every point in $ L( \bn) $ has a limit and so $L( \bn) \subset Z( \bn)$.  This is actually an inclusion of varieties of the same dimension (in each connected component) and thus gives an inclusion $ \Irr L( \bn) \subset \Irr Z( \bn)$.
\begin{theorem}(\cite[Theorem 4.6]{Nak01}.) \label{th:Nak}
  There is a crystal structure on $ \Irr Z( \bn) $, extending the crystal structure on $ \Irr L(\bn) $, such that \cref{eq:bijIrr} is an isomorphism of crystals (with respect to the tensor product crystal structure on the left hand side). 
\end{theorem}
Combining with \cref{eq:fixedpoints}
$$ 
\Irr L(\bn) \subset \Irr Z( \bn) \cong  \Irr L( \bn^1) \times \cdots \times \Irr L( \bn^r)
$$
and we get the crystal inclusion $$ B(\lambda) \subset B(\lambda^1) \otimes \cdots \otimes B(\lambda^r)\,. 
$$
\subsection{Components and filtrations}
\label{sec:compfilt}
Let  $ M \subset T(\lambda) $ 
be a submodule.  For each $k = 1, \dots, r $, we can consider the submodule
$$ M^{\le k} := M \cap \left( T(\lambda^1) \oplus \cdots \oplus T(\lambda^k)\right) $$
This defines a filtration $ M^{\le 1} \subset M^{\le 2} \subset \cdots \subset M^{\le r} $.  
We can regard the subquotients $M^k := M^{\le k} / M^{\le k-1} $ as a submodules of $ T(\lambda^k) $.
\begin{example}\label{ex:jfilt}
  Choose $ \fg = \mathfrak{sl}_4$ and ${\bf n} = (0,3,0) $, with $ \mathbf n^1 = \mathbf n^2 = \mathbf n^3 = (0,1,0)$, corresponding to $3\omega_2 = \omega_2 + \omega_2 + \omega_2$. We label the standard basis vectors of $T(w)^{\oplus 3} = T(\omega_2)^{\oplus 3}$ so that vectors with superscript $k$ lie in the $k$th copy of $T(\omega_2)$ in the direct sum as follows
  $$
  P(w)^{\oplus 3} = \begin{tikzcd}[row sep = tiny, column sep = tiny]
    & \mathbb{C}^3 = \mathrm{span}(v_2^1, v_2^2, v_2^3) \ar[dl]\ar[dr]\\
    \mathbb{C}^3 = \mathrm{span}(u_1^1, u_1^2, u_1^3) \ar[dr] & & \mathbb{C}^3 = \mathrm{span}(w_1^1, w_1^2, w_1^3) \ar[dl] \\
    & \mathbb{C}^3 = \mathrm{span}(v_1^1, v_1^2, v_1^3)
  \end{tikzcd}.
  $$
  Now let
  \begin{align*}
  M = \mathrm{span}( & u_1^1 + u_1^2 + u_1^3, \\
  & v_1^1, v_1^2, v_1^3, v_2^1 + v_2^2 + v_2^3, \\
  & w_1^1 + w_1^2 + w_1^3, 3w_1^3 + 2w_1^2 + w_1^3).
  \end{align*}
  Then we have
  $$
  M^{\le 1} = \mathrm{span}(v_1^1) \subset 
  M^{\le 2} = \mathrm{span}(v_1^1, v_1^2, 2w_1^3 + w_1^2) \subset 
  M^{\le 3} = M,
  $$
  and we can write the subquotients as 
  $$
  M^1 = \begin{tikzcd}[row sep = tiny, column sep = tiny]
    & 0 \ar[dl]\ar[dr]\\
    0 \ar[dr] & & 0 \ar[dl] \\
    & 1
  \end{tikzcd},  ~
  M^2 = \begin{tikzcd}[row sep = tiny, column sep = tiny]
    & 0 \ar[dl]\ar[dr]\\
    0 \ar[dr] & & 1 \ar[dl] \\
    & 1
  \end{tikzcd}, ~
  M^3 = \begin{tikzcd}[row sep = tiny, column sep = tiny]
    &1\ar[dl]\ar[dr]\\
    1\ar[dr] & & 1\ar[dl] \\
    & 1
  \end{tikzcd}.
  $$
\end{example}

\begin{lemma}
  With the above notation, we have $$ \lim_{s \rightarrow 0} M = (M^1, \dots, M^r)$$ in $Gr(T(\lambda^1)) \times \cdots \times Gr(T(\lambda^r)) $.
\end{lemma}
\begin{proof}
    The $ \Cx $-action on the quiver Grassmannian $ \Gr(T(\lambda))$ is the restriction of a $ \Cx $-action on the usual Grassmannian of all (graded) subspaces of $ N $.  With this in mind, the Lemma follows from a standard fact about limits of points in Grassmannians under $ \Cx $ actions.
\end{proof}
From this Lemma and \Cref{th:Nak}, we deduce the following.

\begin{theorem}\label{thm:components}
   Let $ (X_1, \dots, X_r) \in \Irr Gr(T(\lambda^1)) \times \cdots \times \Irr Gr(T(\lambda^r))$.
  Assume that $ X_1 \otimes \cdots \otimes X_r $ lies in $B(\lambda) \subset B(\lambda^1) \otimes \cdots \otimes B(\lambda^r) $.  
  Then the closure $Z(X_1,\dots,X_r)$ of 
  $$  
  \left \{ M \in \Gr(T(\lambda)) : M^k \in X_k, \text{ for $ k = 1, \dots, r $ } \right \}
  $$
  is an irreducible component of  $\Gr(T(\lambda))$.
	Moreover, $(X_1,\dots,X_r)\mapsto Z(X_1,\dots,X_r)$ gives the crystal isomorphism $ B(\lambda) \rightarrow \Irr \Gr(T(\lambda))$.
\end{theorem}

\subsection{Generalization to Demazure crystals}

Now, we will generalize the analysis of the previous section to Demazure tensor product varieties and Demazure crystals.  Fix any $ \lambda \in P_+ $ and any $ w \in W $.  We have an inclusion $ T(\lambda, w) \subset T(\lambda) $ which yields an inclusion of the quiver Grassmannians $ \Gr (T(\lambda, w)) \subset \Gr(T(\lambda)) $.  By \Cref{thm:DemCrystalGeometric}, this gives an inclusion of irreducible components $ \Irr \Gr (T(\lambda, w)) \subset \Irr \Gr(T(\lambda)) $ which corresponds to the crystal embedding $ B_w(\lambda) \subset B(\lambda)$.

On the other hand, we can consider the inclusion $ B(\lambda) \subset B(\lambda^1) \otimes \cdots \otimes B(\lambda^r) $.  Under this inclusion, we see that the Demazure crystal $ B_w(\lambda) $ is sent into the corresponding tensor product of Demazure crystals $ B_w(\lambda^1)\otimes \cdots \otimes B_w(\lambda^r)$ (because the extremal weight element $ b_{w\lambda}$ mapsto $ b_{w\lambda^1} \otimes \cdots \otimes b_{w\lambda^r}$).

Note also that $ T(\lambda,w) = T(\lambda^1, w) \oplus \cdots \oplus T(\lambda^r,w) $ and thus for $ M \in \Gr(T(\lambda,w))$, $ M^k \in \Gr(T(\lambda^k,w))$.

Combining these observations, \Cref{thm:components} actually implies the analogous statement for Demazure modules.
\begin{theorem}\label{thm:componentsDem}
	Let $ (X_1, \dots, X_r) \in \Irr Gr(T(\lambda^1,w)) \times \cdots \times \Irr Gr(T(\lambda^r,w))$ be such that
	$ X_1 \otimes \cdots \otimes X_r $ lies in $B_w(\lambda) \subset B_w(\lambda^1) \otimes \cdots \otimes B_w(\lambda^r) $.  
	Then the closure $Z(X_1,\dots,X_r)$ of 
	$$  
	\left \{ M \in \Gr(T(\lambda,w)) : M^k \in X_k, \text{ for }k = 1, \dots, r \right \}
	$$
	is an irreducible component of  $\Gr(T(\lambda,w))$.
	Moreover, $(X_1,\dots,X_r)\mapsto Z(X_1,\dots,X_r) $ gives the crystal isomorphism $ B_w(\lambda) \rightarrow \Irr \Gr(T(\lambda,w))$.
\end{theorem}

\section{The bijection}
\label{sec:bijection}

\subsection{From RPPs to components}
\label{subsec:rppstocomps}

Fix a dominant minsucule $ w \in W $ with witness $ \lambda \in P_+ $.  Recall that by \Cref{th:GrHw}, $ \Gr(\bC H(w)) $ is the finite set of order ideals of $ H(w)$.

Let $ \Phi \in \Rpp(w,n) $.  Let $ \Phi \mapsto (\phi^1, \dots, \phi^n) $ be the decomposition into order ideals given by \cref{eq:gravsort}.  Define
$$
Z(\Phi)^o :=  \{ M \in \bC H(w)^{\oplus n} : M^k = \bC \phi^k \text{ for $ k = 1, \dots, n $ } \} 
$$
and let $ Z(\Phi) $ be its closure.

We will now apply \Cref{thm:componentsDem} to $T(n\lambda, w)$, with $ r = n $ and each $ \lambda^k = \lambda $.  Each $ T(\lambda^k, w) = \bC H(w) $ (by \Cref{thm:projcover}) and each $ \Gr(T(\lambda^k, w)) $ is a finite set in bijection with the set of order ideals $ J(H(w))$ (by \Cref{th:GrHw}). With the aid of \Cref{thm:gravsort}, \Cref{thm:componentsDem} translates into the following statement.
\begin{theorem} \label{thm:RPPtoIrr}
	For each $ \Phi \in \Rpp(w,n) $, $ Z(\Phi) $ is an irreducible component of $ Gr(\bC H(w)^{\oplus n}) $ and $\Phi \mapsto Z(\Phi)$ gives the crystal isomorphism $ \Rpp(w,n) \cong \Irr \Gr(\bC H(w)^{\oplus n})$.
\end{theorem}

\begin{remark}
Here is another perspective on this result. We have a bijection between $ \Gr(\bC H(w)^{\oplus n})^{\Cx}$ and $ J(H(w))^n $ by \cref{eq:fixedpoints2}.  For each of these fixed points, we can take its attracting set.  \Cref{thm:RPPtoIrr} tells us those fixed points which come from RPPs are exactly those fixed points whose attracting sets are dense in components. 
\end{remark} 

Recall that if $\Phi \in \Rpp(w,n) $, then by \Cref{thm:gravsort}, $ \phi^k  \subseteq \phi^{k+1} $ for all $ k $, and so $ \bC \phi^k \subseteq  \bC \phi^{k+1} $ is a submodule, for all $ k$.  
\Cref{thm:RPPtoIrr} thus has the following surprising consequence. Consider the locus of those $ M$ with increasing subquotients,
$$ 
U := \{ M \in Gr(\bC H(w)^{\oplus n}) :  M^k \subseteq M^{k+1} , \text{ for all $ k $} \}
$$
\begin{corollary}
  $ U $ is dense subset of $ Gr(\bC H(w)^{\oplus n}) $.
\end{corollary}

We are now ready to prove that the map $ Z \rightarrow \Phi_Z $ defined in \cref{sec:nilp} gives a bijection between $ \Irr \Gr(\bC H(w)^{\oplus n}) \rightarrow \Rpp(w,n) $, proving \Cref{thm:mainprecise}.
As we now have a bijection in the opposite direction, in order to prove \Cref{thm:mainprecise}, it suffices to show that if $ \Phi \in \Rpp(w,n) $, then $ \Phi_{Z(\Phi)} = \Phi $.  Now recall that $ \Phi_Z = \Phi_M $ for a general point $ M \in Z $, and $ \sum_{r \le s} \Phi_M(x_i^r) = \dim (M_i \cap \ker A_i^s) $. Thus, we must prove that for all $ M \in Z(\Phi)^o $, for all $ i \in I $ and for all $ s \ge 0 $, we have
\begin{align} \label{eq:dimeq2}
	\dim (M_i \cap \ker A_i^s) & =\sum_{r \leq s} \Phi(x_i^r).
\end{align}

Note that $ Z(\Phi)^o \subset U $. By the definition of $ Z(\Phi)^o $ the subquotient $ M^k $ is determined by $ \phi^k $ and a small calculation shows that
\begin{align*}
\sum_{r \leq s} \Phi(x_i^r) &= \sum_{k=0}^n \dim (M_i^k \cap \ker A_i^s)
\end{align*}
Therefore to establish \cref{eq:dimeq2} for all $ M \in Z(\Phi)^o $, we must prove the following. 
\begin{lemma} \label{th:MMk}
	For all $ M \in U $, we have
	$$
\dim (M_i \cap \ker A_i^s) = \sum_{k=0}^n \dim (M_i^k \cap \ker A_i^s)
$$
\end{lemma}

\begin{example}\label{ex:jfilt2}
  Let $M$ be as in \Cref{ex:jfilt}. Then we see that
  $$
  \phi^1 = \begin{tikzcd}[row sep = tiny, column sep = tiny]
    & 0 \ar[dl]\ar[dr]\\
    0 \ar[dr] & & 0 \ar[dl] \\
    & 1
  \end{tikzcd},  ~
  \phi^2 = \begin{tikzcd}[row sep = tiny, column sep = tiny]
    & 0 \ar[dl]\ar[dr]\\
    0 \ar[dr] & & 1 \ar[dl] \\
    & 1
  \end{tikzcd}, ~
  \phi^3 = \begin{tikzcd}[row sep = tiny, column sep = tiny]
    &1\ar[dl]\ar[dr]\\
    1\ar[dr] & & 1\ar[dl] \\
    & 1
  \end{tikzcd}.
  $$
  In this example, \Cref{th:MMk} is nontrivial only when $i = 2, s = 1$. We verify that 
  \begin{align*}
  \dim ((\mathbb{C}\phi^1)_2 \cap \ker A_2) + \dim ((\mathbb{C}\phi^2)_2 \cap \ker A_2) + \dim ((\mathbb{C}\phi^3)_2 \cap \ker A_2) & = 1 + 1 + 1\\
& = \dim(M_2\cap \ker A_i).
  \end{align*}
  If we choose, however, an arbitrary module $M$ (not necessarily lying in some $Z(\Phi)^\circ$), then \Cref{th:MMk} can fail.  
  Indeed, keep $w$, label the standard basis vectors of $\mathbb{C}H(w)^{\oplus 2}$ as in
  $$\begin{tikzcd}[row sep = tiny, column sep = tiny]
    &\mathbb{C}^2 = \mathrm{span}(v_2^1, v_2^2)\ar[dl]\ar[dr]\\
    \mathbb{C}^2 = \mathrm{span}(u_1^1, u_1^2)\ar[dr] & & \mathbb{C}^2 = \mathrm{span}(w_1^1, w_1^2)\ar[dl] \\
    & \mathbb{C}^2 = \mathrm{span}(v_1^1, v_1^2)
  \end{tikzcd},$$
  and consider the submodule $M = \mathrm{span}(u_1^1, v_1^2 + v_2^1, v_1^1, w_1^1)$. It has 
  \begin{align*}
  M^{\leq 1} & = \mathrm{span}(u_1^1, v_1^1, w_1^1), & M^{\leq 2} & = M,\\
  M^1 & \cong M^{\leq 1}, & M^2 & = M^{\leq 2}/M^{\leq 1} \cong \mathrm{span}(\overline{v_1^2 + v_2^1}).
  \end{align*}
  This means 
  $$\dim(M_2^1\cap \ker A_2) + \dim(M_2^2\cap \ker A_2)= 1 + 1 = 2 > 1 = \dim(M_2 \cap \ker A_2).$$
\end{example}

To prove \Cref{th:MMk}, we use the following two technical conditions which a module $M$, a vertex $i \in I$, an integer $k\in [0,n] $, and an integer $s \geq 0$ might satisfy:
\begin{align}
\parbox[t]{11cm}{If there exists $m \in M_i^{\le k-1}$ for which $A_i^s(m) \not\in M_i^{\le k-2}$, then there also exists $m' \in M_i^{\le k}$ such that $A_i^s(m') \not\in M_i^{\le k-1}$.} \tag{$C1(i, k, s)$}
\end{align}
\begin{align}
\parbox[t]{11cm}{If there exists $m \in M_i^{\le k}$ for which $A_i^s(m) \neq 0$, then there also exists $m' \in M_i^{\le k}$ such that $A_i^s(m') \not\in M_i^{\le k-1}$.} \tag{$C2(i, k, s)$}
\end{align}
We write $C1$ and $C2$ to denote the above conditions for all values of $i, k, s$. It is clear that if a module satisfies $C2$, then it satisfies $C1$, but the converse is not clear. The purpose of the next subsection will be to show that any module satisfying $C1$ must also satisfy $C2$, thereby proving that the two properties are equivalent. These conditions are motivated by the following sequence of results.

\begin{lemma} \label{lem:uc1}
Any $M \in U$ satisfies $C1$.
\end{lemma}

\begin{proof}
This is immediate from the definition of $ U $.
\end{proof}

\begin{lemma} \label{lem:c1c2}
	If $ M $ satisfies $ C1 $, then it satisfies $ C2$.
\end{lemma}

\Cref{sec:c1c2} is dedicated to the proof of this lemma.

\begin{lemma} \label{lem:c2dim}
  If $M$ satisfies $C2$, then \Cref{th:MMk} holds for $M$.
\end{lemma}

\begin{proof}
	By lifting bases from the subquotients $ M_i^k $, it is enough to show that any $\overline{m} \in M_i^k$ such that $A_i^k(\overline{m}) = 0$, has a lift $m \in M_i^{\le k}$ such that $A_i^k(m) = 0$.
	
	Suppose for contradiction that this fails for some $i$. Choose the maximal $k$ such that it fails, i.e. there exists $\overline{m} \in M_i^k$ for which $A_i^k(\overline{m}) = 0$, but has no lift $m$ such that $A_i^k(m) = 0$. Let $m'$ be some lift; then by assumption, $A_i^k(m') \neq 0$. So by (C2), there exists some $m'' \in M_i^{\le k}$ such that $A_i^k(\overline{m''}) \neq 0$. Let $s > k$ be the minimum number such that $A_i^s(\overline{m''}) = 0$. By maximality of $k$, there is some lift $m'''$ of $\overline{m''}$ with $A_i^s(m''') = 0$. Replacing $m''$ by $m'''$ if necessary, we can assume without loss of generality that $A_i^s(m'') = 0$ (but $A_i^{s-1}(\overline{m''}) \neq 0$).
	
	The fact that $A_i^{s-1}(\overline{m''}) \neq 0$ means that the set of $k$ vectors $\{A_i^{q}(\overline{m''})\}_{q=s-k}^{s-1}$ span the entire space $\ker \overline{A_i^k} \subset M_i^k$ (this can be seen by an easy induction). This means there is some linear combination $m$ of these vectors $A_i^q(m'')$ which is a lift of $\overline{m}$. But since $m \in \text{span} \{A_i^{q}(m'')\}_{q=s-1}^{s-k}$, we see that $m \in \ker A_i^k$, which is a contradiction, since we assumed no such lift of $\overline{m}$ existed.
\end{proof}

These three lemmas, \Cref{lem:uc1,lem:c1c2,lem:c2dim}, complete the proof of \Cref{th:MMk} and the proof of \Cref{thm:mainprecise}.

\subsection{C1 implies C2}\label{sec:c1c2}
Throughout this section fix a dominant minuscule heap $ H(w) $ and $ M \in Gr(\bC H(w)^{\oplus n})$.

We now show in this section that if $M$ satisfies $C1$, it must also satisfy $C2$; this is the most technical part of our argument, since it relies on the underlying combinatorics of the heap corresponding to $\bC H(w)$. We begin with some combinatorial definitions which will be useful for the proof.

\begin{definition}
Fix a single arrow $i \to i'$ in the quiver, and fix neighbouring beads $a > b$ on the runner $i$. By Proposition 2.5 in \cite{St01}, the local behaviour of $ \bC H(w)$ is given by one of the following figures in \Cref{fig:types}. We use solid arrows to denote arrows pointing from vertex $i$ to vertex $i'$, and dotted arrows for all the others. Solid dots represent beads adjacent to either $a$ or $b$ which lie on a runner which is not $i$.  The adjacent beads $a, b$ are said to be of Type $1$, $2$, and $3$ with respect to the arrow $i \to i'$ as specified in \Cref{fig:types}.
\begin{figure}[ht]
\caption{Neighbouring vertices $a, b$ of Type 1, 2, and 3 respectively, with respect to $i \to i'$. In these figures, $a$ and $b$ lie on runner $i$ while $i'$ is the label of the neighbouring runner pictured to the right of $i$.\label{fig:types}}
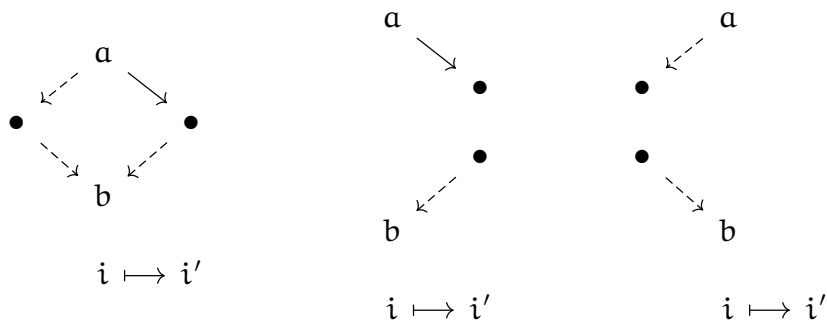
\[
\begin{tikzcd}[column sep = small, row sep = small]
& a \ar[rd] \ar[ld, dashed] & \\
\bullet \ar[rd, dashed] & & \bullet \ar[ld, dashed] \\
& b & \\
& i \ar[r, maps to] & i'
\end{tikzcd}
\hspace{1.5cm}
\begin{tikzcd}[column sep = small, row sep = small]
& a \ar[rd] & \\
& & \bullet \\
& & \bullet \ar[ld, dashed] \\
& b & \\
& i \ar[r, maps to] & i' 
\end{tikzcd}
\hspace{1.5cm}
\begin{tikzcd}[column sep = small, row sep = small]
& a \ar[ld, dashed] & \\
\bullet & & \\
\bullet \ar[rd, dashed] & &  \\
& b & \\
& i \ar[r, maps to] & i'
\end{tikzcd}
\]
\end{figure}
\end{definition}

With this definition in hand, we now define a special property which a vertex $i$ might have. For this definition, we will use the notation $d(i, i')$ to denote the distance between any two vertices in $I$. Note also that for any $i, i' \in I$ distinct, there exists some unique $i''$ adjacent to $i'$ such that $d(i, i'') < d(i, i')$. We call $i''$ the \textbf{$i$-toward vertex adjacent to $i'$}. 

\begin{definition}
We say that $i \in I$ is \emph{attractive} if the following is true: given any $i' \in I$ and letting $i''$ be the $i$-toward vertex adjacent to $i'$, any two adjacent beads $a > b$ on runner $i'$ are of Type 1 or Type 2 with respect to $i' \to i''$.
\end{definition}

\begin{lemma}\label{lem:onebead}
	If there is only one bead on the runner $ i $, then if property $C1(i, k, s)$ holds for $M$ and some choice of $k, s$, then $M$ also satisfies $C2(i, k, s)$.
\end{lemma}
\begin{proof}
First note that the statement only needs to be checked for $ s = 0 $, since $ A_i^s = 0 $ for $ s > 0 $.	
	
If $ C1(i, k, s)$ holds, we know that $ \dim M_i^{\le q}/ M_i^{\le q-1} $ is non-decreasing as a function of $q \in {1, \dots, k}$, and since there is only one bead on runner $ i $, it is always at most $1$.  Choose $ k' \leq k$ maximal such that $ \dim M_i^{\leq k'}/ M_i^{\leq k' - 1} = 0 $.  This implies that $ M_i^{\leq k'} = 0 $, since $ \dim M_i^{\leq k'} = \sum_{r \le k'}  \dim M_i^{\leq r}/ M_i^{\leq r-1} $.

Now, let $  m \in M_i^{\le k} $, $ m \ne 0 $.  This means $ k > k' $ and thus $  \dim M_i^{\le k}/ M_i^{\le k-1} = 1 $.  Hence there exists $ m' \in M_i^{\le k} $ with $ m' \notin M_i^{\le k-1} $ as desired.
\end{proof}

In the remainder of this section, we will make use of the notion of ``slant-irreducibility" of a heap, as discussed in \cite{St01} and \cite{P992}. We briefly review the setup now.

\begin{definition}
Given a dominant minuscule heap $H$ relative to a connected Dynkin diagram, let $T$ be the poset of elements of $H$ consisting of the top bead on each runner. It is shown in \cite{St01} that $T$ is a tree; we call $T$ the \emph{top tree} of $H$.
\end{definition}

\begin{definition}\label{def:slantsum}
Let $H_1$ and $H_2$ be heaps of dominant minuscule elements supported on two disjoint (but connected) simply-laced Dynkin diagrams.  Let $p$ be a bead in the top tree of $H_1$ whose label $i$ occurs only once in $H_1$ (i.e. $p$ is the only bead on the $i$th runner). Let $q$ be the maximum element of $H_2$, and let $j$ be the vertex over which it lies. Then Proposition 3.1 in \cite{St01} gives that the labelled poset obtained from $H_1 \cup H_2$ by adding the covering relation $q < p$ is a dominant minuscule heap relative to any Dynkin diagram obtained by taking the union of the two original diagrams and adding an edge between $i$ and $j$. We call this new poset the \emph{slant sum} of $H_1$ and $H_2$.

If $H$ is a dominant minuscule heap which cannot be written as a slant sum of two other such heaps, we say $H$ is \emph{slant irreducible}.
\end{definition}

The following is then proved in \cite{P992} and discussed in the context of dominant minuscule heaps in \cite{St01}.
\begin{lemma}
Any dominant minuscule heap $H$ can be written as an iterated slant sum of slant-irreducible dominant minuscule heaps.
\end{lemma}

This will be an advantage going forward, since slant-irreducible dominant minuscule heaps over connected simply-laced Dynkin diagrams were fully classified by Proctor in \cite{P992}.

\begin{proposition}\label{prop:extremal}
For $H(w)$ a slant-irreducible dominant minuscule heap, there always exists some attractive vertex $i$ such that if $M$ satisfies $C1$, then $C2(i, k, s)$ holds for all $k, s$ and for this specific choice of $i$.
\end{proposition}
\begin{proof}
By the classification of slant-irreducible dominant minuscule heaps in \cite{P992} (translated from the language of ``$d$-complete posets" to the language of heaps by the last section of loc. cit. and \cite{P991}), we can split into cases depending on the Dynkin type. During the proof, we will make reference to the enumeration of the 15 classes of posets described in Section 7 of \cite{P992}.

\textbf{Type A.} The only class of posets which can occur here is Class 1 of \cite{P992}, an example of which is pictured in Figure \ref{fig:typedposets} for Type $A_5$. Let $i$ be any extremal vertex; say $i = 1$ for convenience. Observe that for any minuscule poset in Type A, any extremal vertex supports only a single bead, so we can apply Lemma \ref{lem:onebead}. Now note that in Type A, all neighbouring vertices on any runner are of Type 1 with respect to any arrow,  and so all vertices (including $i$) are attractive in this case.

\textbf{Type D.} Two classes of heaps can occur here: Class 2 and Class 4 (with $h = 1$) from \cite{P992}. These are pictured in Figure \ref{fig:typedposets} for Type $D_n$ where $n = 5$. For each of these classes, there exists some extremal vertex which supports only a single bead. Indeed, in Class 2, vertex $1$ has this property and is attractive, and in Class 4, either of the vertices $n-1$ or $n$ has this property and each is attractive, and so we cna once again apply Lemma \ref{lem:onebead}.

\textbf{Type E.} So far, we see that the verification of this proposition has been simple and almost type-independent. This is not the case, to our knowledge, in Type $E$. Classes 3 through 15 of \cite{P992} can occur in Type $E$. It is an easy check that in each of these cases, the trivalent vertex is always attractive; let us denote by $i_*$ its label.

What remains is to manually check that $C1$ implies $C2(i_*, k, s)$ for any $k, s$, which will complete the proof of the proposition. In contrast to the simple arguments in Types A and D explained above, checking this manually for the posets in Classes 3 through 15 of \cite{P992} is a tedious combinatorial and linear-algebraic computation. Nevertheless, it can be done by hand by contradiction, splitting into cases by considering the smallest value of $s$ for which $C2(i, k, s)$ fails. Since it uses the same techniques as we will introduce in the proof of \Cref{lem:c1c2} below
and is unenlightening (except for its indication of the complexity of the problem, which we believe is inherent to the combinatorics of minuscule posets), we leave the explicit check to the reader.\footnote{As a sample of the way the interested reader might carry out this check, we note that in the case of the poset $F(E_6, 5)$ in \cite{W03}, one way to proceed is to reduce the $i = 3$ case to the $i = 0$ case, and then make use of the sequences of arrows $0 \to 3 \to 2$, $2 \to 3 \to 4 \to 5$, $5 \to 4 \to 3 \to 0$, and $2 \to 3 \to 0$ where appropriate, splitting into cases when necessary and using the techniques introduced in the proof of \Cref{lem:c1c2}
to reason along the way. Note also that $F(E_6, 5)$ and $F(E_6, 1)$ are symmetric, so it suffices to consider only the former. In checking the result for $F(E_7, 6)$, we found it useful to use the $E_6$ result and exploit the fact that this poset looks like a gluing of $F(E_6, 1)$ and $F(E_6, 5)$.}
\end{proof}


\begin{proof}[Proof of \Cref{lem:c1c2}]
Suppose for contradiction that $M$ satisfies $C1$, but fails to satisfy $C2$.  Let $k$ be the minimal value for which $C2(i, k, s)$ fails for some $i, s$.  Without loss of generality, for the remainder of the proof we can assume $H(w)$ is slant-irreducible by replacing the original heap $H(w)$ by the slant-irreducible component of $H(w)$ containing $i$. By \Cref{prop:extremal}, there exists some attractive vertex $i_*$ in this same component for which $C2(i_*, k',s)$  hold for all $k', s$. Our assumption of slant-irreducibility poses no issue for the remainder of the proof, since from now on we will only make use of vertices between $i$ and $i_*$ and the arrows between them in the Dynkin quiver, all of which lie in the same slant-irreducible component by our assumption on $i_*$ (and since slant sum is an operation between heaps supported on disjoint subsets of $I$; see \Cref{def:slantsum}). Now Choose $i \in I$ such that $C2(i, k, s)$ fails for some $s$, and such that $d(i, i_*)$ is minimal.

Since $d(i, i_*) \geq 1$ by \Cref{prop:extremal}, there exists a choice of single arrow $a$ in the preprojective algebra for which $d(a(i), i_*) < d(i, i_*)$. (By abuse of notation, we use the same notation for the arrow $a$ of the doubled quiver as we do for the arrow $a$ as an element of the preprojective algebra). Since $C2(i, k, s)$ fails for some $s$ by assumption, let $s \geq 1$ be such that there exists some $m \in M_i^{\le k}$ for which $A_i^s(m) \neq 0$, but such that $A_i^s(M_i^{\le k}) \subset M_i^{\le k-1}.$ (Note that if $C2(i, k, s)$ fails already for $s = 0$, then this is a contradiction by the same argument as in \Cref{lem:onebead}). Let $b_s$ be the bead on runner $i$ corresponding to $\ker A_i^s \setminus \ker A_i^{s-1}$, and let $b_{s+1}$ be the runner above it, corresponding to $\ker A_i^{s+1} \setminus \ker A_i^s$. Both of these beads exist on the runner $i$ since $A_i^s(M^{\le k}) \neq 0$ for our $s \geq 1$.

By the fact that $i_*$ is attractive, we know that $b_{s}$ and $b_{s+1}$ are neighbouring vertices of Type 1 or Type 2 with respect to the arrow $a$. In particular this means $a(b_{s+1})$ is not zero, and is some bead $c_{r+1}$ on runner $a(i)$ corresponding to $\ker A_{a(i)}^{r+1} \setminus \ker A_{a(i)}^r$, for some $r \geq 0$. Algebraically, this means that $ a $ gives a morphism $ \ker A_i^{s+1}/ \ker A_i^s \rightarrow \ker A_{a(i)}^{r+1} / \ker A_{a(i)}^r $; it respects the filtration $M_i^{\leq q}$, thus $A_{a(i)}^r(a(m)) \neq 0$, but $A_{a(i)}^r(a(m)) \in M_{a(i)}^{\le k-1}$. 

Now let $l$ be the maximum value for which there exists $m' \in M_{a(i)}^{\le k}$ for which $A_{a(i)}^l(m') \neq 0$. By the above, we have that $l \geq r$. Now recall that $d(a(i), i_*) < d(i, i_*)$, so by the minimality of $ i $, $C2(a(i), k, s)$ must hold for all $s$.  Hence there exists some $m'' \in M_{a(i)}^{\le k}$ such that $A_{a(i)}^{l}(m'') \not\in M_{a(i)}^{\le k-1}$. By the maximality of $l$, we also have $A_{a(i)}^{l+1}(m'') = 0$. Let $c_{l+1}$ be the bead on runner $a(i)$ corresponding to $\ker A_{a(i)}^{l+1} \setminus \ker A_{a(i)}^{l}$.

Since $l \geq r$, we have $c_{l+1} \geq c_r$, and by the picture for neighbouring vertices of Type 1 and Type 2, we see that $c_{r+1} > b_{s}$ in the poset order on our heap. So $c_{l+1} > b_s$, which means there is some path $p$ in the preprojective algebra with $p(c_{l+1}) = b_s$. Algebraically, this means that $ p $ gives a (filtration-respecting) morphism $ \ker A_{a(i)}^{l+1}/ \ker A_{a(i)}^l \rightarrow \ker A_{i}^{s+1} / \ker A_{i}^s $. Since $A_{a(i)}^{l+1}(m'') = 0$ but $A_{a(i)}^l(m'') \not\in M_{a(i)}^{\le k-1}$, we have that $A_i^s(p(m'')) = 0$, but $A_i^{s-1}(p(m'')) \not\in M_i^{\le k-1}$.

It is clear from the fact that $A_i^{s-1}(p(m'')) \not\in M_i^{\le k-1}$ that the set of $s$ vectors $\{\overline{A_i^q(p(m''))}\}_{q=0}^{s-1}$ considered as elements of $M_i^k = M_i^{\le k}/M_i^{\le k-1}$ span the entire space $\ker \overline{A_i^{s}}$, again considered as a subspace of $M_i^k$. This means that there is some linear combination $t$ of the vectors $A_i^q(p(m''))$, which lies in $M_i^{\le k} \cap \ker A_i^{s}$, such that $\overline{t} = \overline{m}$ in the quotient $ M_i^k$. This means $m - t$ is an element of $M_i^{\le k-1}$ satisfying $A_i^s(m - t) = A_i^s(m) \neq 0$.

Since $k$ is the minimum $k$ for which $C2(i, k, s)$ fails, we can use the property $C2(i, k-1, s)$ to conclude that there exists some $v \in M_i^{\le k-1}$ such that $A_i^s(v) \not\in M_i^{\leq k - 2}$. Subsequently we can use the existence of such a $v$ to apply $C1(i, k-1, s)$, which gives us the existence of some $v' \in M_i^{\le k-1}$ such that $A_i^s(v') \not\in M_i^{\le k-1}$. But in the beginning, we assumed that $A_i^s(M_i^{\le k}) \subset M_i^{\le k-1}$; this is a contradiction.
\end{proof}

\bibliographystyle{alpha}

\end{document}